\documentclass[a4paper,12pt]{article}

\usepackage{amssymb}
\usepackage{epsfig}
\usepackage{amsfonts}
\usepackage{amsmath}
\usepackage{euscript}
\usepackage{amscd}
\usepackage{amsthm}
\usepackage{theoremref}
\usepackage{enumerate}
\usepackage{array}
\usepackage{mathrsfs}
\usepackage{tikz-cd}
\usepackage{graphicx}
\usepackage{tikz}
\usetikzlibrary{calc}
\DeclareMathAlphabet{\mathpzc}{OT1}{pzc}{m}{it}
\DeclareMathAlphabet{\mathpzc}{OT1}{pzc}{m}{it}

\newtheorem{thm}{Theorem}[section]
\newtheorem{lem}[thm]{Lemma}
\newtheorem{prop}[thm]{Proposition}

\newtheorem{rem}[thm]{Remark}

\newtheorem{Defn}[thm]{Definition}

\newcommand{\Z}{\mathbb Z}

\newcommand{\ZZ}{\Z_{\geqslant 0}}
\newcommand{\ZY}{\Z_{\geqslant 1}}
\newcommand{\x}{x_1,\dots,x_m}

\newcommand{\kk}{\mathbb{K}}

\title{ 
Regular Borel subalgebras of\\
 the Lie Algebra of Aut($\mathbb{A}^2$)
}
\author{
		Ananya Pal\\
		{\small{\it Faculty of Computer Science, HSE University,}}\\ 
		{\small{\it Pokrovsky Boulevard 11, Moscow
				109028, Russia}}\\
		{\small {\it e-mail : apal@hse.ru,
				palananya1995@gmail.com }}\\
}

\begin{document}
	\date{}
	\maketitle
	
		\abstract
		In this paper, we find all regular Borel subalgebras of Lie(Aut($\mathbb{A}^2$)), i.e., maximal solvable subalgebras generated by homogeneous derivations with respect to the standard $\mathbb{Z}^2$-grading. 
		 It follows that a regular Borel subalgebra has derived length $2$ or $3$.
		We also describe isomorphism classes of such subalgebras.
		
		On the way, we give a combinatorial description of the regular Abelian subalgebras and the regular solvable subalgebras of Lie(Aut($\mathbb{A}^2$)).

		
		\smallskip
		
		\noindent
		{\small {{\bf 2020 MSC}. Primary:  
				17B30, %(Solvable, nilpotent (super)algebras)
17B66; %(Lie algebras of vector fields and related (super) algebras) 
Secondary: 
17B65, %(Infinite-dimensional Lie (super)algebras)
 17B05, %(Structure theory for Lie algebras and superalgebras)
17B70.%(Graded Lie(super)algrbras)

%$k$ replaced by $\mathbb{K}$...$k$ does not look good.			
\section{Introduction:}\label{I}
Throughout the paper, $\kk$ denotes a field of characteristic zero and $\mathbb{A}^n$ denotes the affine space over $\kk$ of dimension $n$.
The symbols $x_1,\dots,x_n,x,y$, etc., will denote indeterminates over the field $\kk$.

Recall that a  Borel  subalgebra of a Lie algebra $\mathscr{G}$ is a maximal solvable Lie subalgebra	of $\mathscr{G}$ (with respect to inclusion of subalgebras).
A {close} topic to Borel subalgebra is Borel subgroup.
A Borel subgroup of an algebraic group~$G$ is a maximal Zariski-closed, connected and solvable algebraic subgroup of $G$.
On the other hand, a Borel subgroup of an ind-group $\mathfrak{G}$ is a maximal solvable connected subgroup of $\mathfrak{G}$.
%The study of Borel subalgebra (resp. Borel subgroup) first of finite dimensional Lie algebras (resp. algebraic group) and now of infinite dimensional Lie algebras (resp. ind-group) has been an important topic of research from its introduction in the middle of the twentieth century.
The study of Borel subgroups and Borel subalgebras are very important in their own right.
In a way, they are the building blocks of group (algebraic group and ind-group) and Lie algebra, respectively.

Sometimes, the study of Borel subgroups and Borel subalgebras are intertwined.
The most well known result in this context is the following: there is a natural bijection between the Borel subgroups of a connected affine algebraic group $G$ over $\kk$ and the Borel subalgebras of Lie(${G}$) \cite[Theorem~13.1]{HL2}. 
Moreover, in this case, all the Borel subgroups are conjugate \cite[Theorem~21.3]{HL2}, and all the Borel subalgebras are Ad-conjugate. 
The latter is also true for all finite dimensional Lie algebras over $\kk$.
In particular, all the Borel subalgebras of a finite dimensional Lie algebra over $\kk$ are conjugate under the action of the automorphism group \cite[Theorem~16.4]{HL}.
However, over a field of positive characteristic this is not always true. 
For example, over a field of characteristic $p> 3$, the first Witt Lie algebra $W_1$ has precisely two conjugacy classes of Borel subalgebras \cite[Theorem 4.2]{YC}.
For conjugacy classes of Borel subalgebras of $W_n$ over a field of characteristic $p>3$, one can see \cite[Theorem~4.6]{BS}.

Now we turn our attention to the ind-groups and their Lie algebras. 
The automorphism group Aut($\mathbb{A}^n$), for $n\geqslant 2$, has a natural affine ind-group structure~\cite{FK}. 
The corresponding Lie(Aut($\mathbb{A}^n$)) is a subalgebra of the Lie algebra of derivations Der($\kk[x_1,\dots,x_n]$). 
Moreover, Der($\kk[x_1,\dots, x_n]$)~=~Vec($\mathbb{A}^n$) consists of all polynomial vector fields over $\mathbb{A}^n$ and the Lie algebra Lie(Aut($\mathbb{A}^n$)) = Vec$^c(\mathbb{A}^n)$ consists of polynomial vector fields over $\mathbb{A}^n$ with constant divergence \cite[Proposition 15.7.2]{FK}. 
{}The above one to one correspondence between Borel
subgroups of algebraic groups and Borel subalgebras of their corresponding Lie algebras is generally not satisfied for ind-groups.} 
One such example is the ind-group Aut($\mathbb{A}^2$).
By \cite[Theorem~1]{BEE} it is known that every Borel subgroup of Aut($\mathbb{A}^2$) is conjugate to the triangular (de Jonqu\'eres) subgroup JONQ$^{\pm}(\mathbb{A}^2)$ of Aut($\mathbb{A}^2$).
In fact, Lie subalgebras corresponding to the triangular subgroups (triangular subalgebras $\mathfrak{j}_2^{\pm}$) are Borel subalgebras of Lie(Aut($\mathbb{A}^2$)) \cite[Proposition~4.9]{AZ3} and they are pairwise conjugate
under the adjoint action of Aut($\mathbb{A}^2$). 
Recently, Arzhantsev and Zaidenberg in \cite[Remark 4.7]{AZ3}, have shown the existence of Borel subalgebra of Lie(Aut($\mathbb{A}^2$)) not
corresponding to any Borel subgroup of Aut($\mathbb{A}^2$).

In this paper, we study the Borel subalgebras of Lie(Aut($\mathbb{A}^2$)) generated by homogeneous derivations (with respect to the standard $\mathbb{Z}^2$-grading) and name them {\it regular Borel subalgebras}.
They can be considered the next non-trivial class after Borel subalgebras corresponding to the Borel subgroups of Lie(Aut($\mathbb{A}^2$)).
We give a structure theorem for them (\thref{MT1}).
Precisely, we prove the following (For notation check Section~\ref{N}~notation 21):

\medskip
\noindent
{\bf Theorem}:
\begin{enumerate}[\rm(a)]
	\item Every regular Borel subalgebra of Lie(Aut($\mathbb{A}^2$)) has derived length $2$ or $3$. 
	\item The only metabelian regular Borel subalgebra is of the form $\mathscr{G}(0,0,(i+1)_i,\pmb{0},\pmb{0})$.
	\item The regular Borel subalgebras of derived length $3$ are of the following form:
	\begin{center}
		$\mathscr{G}(n,1,\pmb{\lambda},\pmb{\alpha},\pmb{\gamma})$, for some $n\in \ZY\cup\{-1\}$	
	\end{center}
	and some non-zero $\pmb{\lambda}, 
	\pmb{\alpha}, \pmb{\gamma}\in$ Map($\ZY,\kk$)
	 (specified in \thref{IC0}).
\end{enumerate}

\smallskip
\noindent
Here, for convenience we presented the result in abstract Lie algebra setup.
However, the statement of \thref{MT1} gives an implicit description (elements being homogeneous derivations) of regular Borel subalgebras.

As a corollary of this result, given any homogeneous derivation $\partial$ we have identified all the regular Borel subalgebras containing $\partial$.
As a consequence,  we get explicit example of Borel subalgebras of Lie(Aut($\mathbb{A}^2$)) not corresponding to any Borel subgroups of Aut($\mathbb{A}^2$) (\thref{MT3}).
Moreover, existence of non-regular Borel subalgebras of Lie(Aut($\mathbb{A}^2$)) follows (\thref{MT2}). 

We end our discussion by studying the isomorphism classes of regular Borel subalgebras.
In fact, in \thref{IC2}(a), we have shown that every regular Borel subalgebra of derived length~$3$ is either isomorphic to $\mathscr{G}(-1,1,({2-i})_i,({-i})_i,(i-1)_i)$ (regular Borel subalgebras corresponding to Borel subgroups) or $\mathscr{G}(n,1,(\frac{i-2}{n})_i,(\frac{i}{n})_i,(i)_i)$, for some $n\in \ZY$ (regular Borel subalgebras not corresponding to Borel subgroups).
Moreover, we also proved that
these $\mathscr{G}'$s
%$(n,1,(\frac{i-2}{n})_i,(\frac{i}{n})_i,(n+i)_i)$ 
are graded non-isomorphic  Lie algebras for different $n$ (cf. \thref{IC2}(b)).

% and also study their isomorphism classes (\thref{IC1}).

Below we give a layout of this paper.
%Let us now discuss the content of this paper.
The next section contains all the notations.
% used throughout the paper.
Then in Section~\ref{P}, we show an important result: the set of all regular Borel subalgebras of Lie(Aut($\mathbb{A}^2$)) coincides with the set of all maximal regular solvable subalgebras (\thref{RBS4}).
This paves our way to prove the structure theorem mentioned above.
So in the next part of the paper, we study {\it regular Abelian subalgebras} (in Section~\ref{RAS}) and {\it regular solvable subalgebras} (in Section~\ref{RSS}).
%Then we introduce regular Abelian subalgebra (in Section~\ref{RAS}) and regular solvable subalgebra (in Section~\ref{RSS}) of Lie(Aut($\mathbb{A}^2$)) and study them.
We observe that maximal regular Abelian subalgebras of Lie(Aut($\mathbb{A}^2$)) are also maximal Abelian subalgebras 
%with respect to the inclusion of subalgebras 
(Propositions~\ref{RAS2} and \ref{RAS3}). 
We also provide a criterion of solvability for regular subalgebras (\thref{GC4}).
In Section~\ref{RBS}, we prove the theorem mentioned above (\thref{MT1}) and write the regular Borel subalgebras in the notation involving $\mathscr{G}$ (\thref{IC0}).
Finally, in Subsection~\ref{IC}, we study the isomorphism classes of the regular Borel subalgebras.
%
%% about the structure of regular Borel subalgebras of Lie(Aut($\mathbb{A}^2$)).
%In order to state the main theorem clearly, we introduce a notation.
%
%For $n\in \ZZ\cup\{-1\}$, $\mu\in \{0,1\}$ and  $\pmb {\lambda}=(\lambda_i)_i, 
%\pmb{\alpha}=(\alpha_i)_i, \pmb{\gamma}=(\gamma_i)_i\in$ Map($\ZY,\kk$),  $\mathscr{G}(n,\mu,\pmb{\lambda},\pmb{\alpha},\pmb{\gamma})$ defines
%%\alpha=(\alpha_i), \gamma=(\gamma_i)\in$ Map($\ZY,\kk$),  $\mathscr{G}(n,\lambda,\alpha,\gamma)$ : on the vector space 
%the Lie algebra on the vector space
%$$V=\kk Y_1\oplus \kk Y_2\oplus\big(\overunderset{\infty}{i=0}{\oplus}\kk X_i\big)$$ as follows:
%\begin{center}
%	$[Y_1,X_0]=X_0,[Y_1,X_i]=\lambda_i X_i,[Y_2,X_0]=\mu X_0$,$[Y_2,X_i]=\alpha_i X_i$,$[X_0,X_i]=\gamma_i X_{i+n},\,\forall i\geqslant 1$ 
%\end{center}
%%where $\lambda=(\lambda_i)_i,\, \alpha=(\alpha_i)_i$ and $\gamma=(\gamma_i)_i$
%and all other commutators are zero.
%In particular, 
%\begin{center}
%
%$\mathfrak{j}_2^{+}$ (or $\mathfrak{j}_2^{-}$) $\cong \mathscr{G}(-1,1,(2-i)_i,(-i)_i,(i-1)_i)$.
%%and 
%%$\mathfrak{j}_2^{-}\cong \mathscr{G}(-1,1,(2-i)_i,(-i)_i,(i-1)_i)$.
%	
%\end{center}

%
%The main result of this paper
%Moreover, by \thref{C1}, it follows that every regular Borel subalgebra $\mathfrak{B}$ of Lie(Aut($\mathbb{A}^2$)) is of the following form:
%\begin{center}
%	$\mathfrak{B}=$Lie($\mathfrak{t}_2,\mathcal{S}$)
%\end{center}

\section{Notations:}\label{N}
This section consists of all the notations used throughout the paper and some small connected remarks.
\begin{enumerate}[\rm 1.]
\item $\mathbb{K}$ : field of characteristic zero

\item $\mathbb{A}^n$ : the affine space over $\mathbb{K}$ of dimension $n\in \ZY$

\item  For any square matrix $A$, $\begin{vmatrix}
	A
\end{vmatrix}$ : determinant of $A$

\item  $\mathcal{M}$ := $\mathbb{Z}_{\geqslant 0}^2\cup\{(m,-1)\mid m\in \mathbb{Z}_{\geqslant 0}\}\cup\{(-1,n)\mid n\in \mathbb{Z}_{\geqslant 0}\}\subseteq \mathbb{Z}^2$
	
\item For any $(a,b)\in \mathbb{Z}^2\setminus \{(0,0)\}$, $\ell_{(a,b)}$ : the line passing through $(0,0)$ and $(a,b)$, i.e.,
\begin{center}
	$\ell_{(a,b)}=\{(c,d)\mid ad-bc=0\}=\{(c,d)\text{ such that } \begin{vmatrix}
		a & c\\
		b & d
	\end{vmatrix}=0\}.$
\end{center}
\end{enumerate}
\begin{rem}\thlabel{N1}
	{\rm For any $(a,b)\in \ZZ^2\setminus\{(0,0)\}$, we know that $\ell_{(a,b)}\cap \mathcal{M}$ is a
	ray in $\ZZ^2$.
		Note that $({a}/{\gcd(a,b)},{b}/{\gcd(a,b)})$ is the minimum element of $\ell_{(a,b)}\cap\mathcal{M}\setminus\{(0,0)\}$ with respect to the lexicographic ordering,
%		 prioritizing the first component.
Therefore, $$\ell_{(a,b)}\cap \mathcal{M}=\bigg\{\bigg(\dfrac{ia}{\gcd(a,b)},\dfrac{ib}{\gcd(a,b)}\bigg)\mid i\in \ZZ\bigg\}.$$
}
\end{rem}

\begin{enumerate}\addtocounter{enumi}{5}
\item For $(a,b)\in \mathbb{Z}^2\setminus\{(-1,-1)\}$, $\ell^{(a,b)}$ : the line passing through $(-1,-1)$ and $(a,b)$, i.e., 
	$$\begin{array}{lll}
	\ell^{(a,b)}&=&\{(c,d)\mid (a+1)(d+1)-(b+1)(c+1)=0\}\\
			&=&	\{(c,d)\text{ such that }
			 \begin{vmatrix}
			a+1 & c+1\\
			b+1 & d+1
		\end{vmatrix}=0\}.
			\end{array}$$

\item For $(a,b)\in \mathbb{Z}^2\setminus\{(0,0)\}$, $\underline{\ell}^{(a,b)}$ : the line parallel to $\ell_{(a,b)}$  passing through $(-1,-1)$, i.e.,
	\begin{center}
		$\underline{\ell}^{(a,b)} =\{(c,d)\mid a(d+1)-b(c+1)=0\}
		=	\{(c,d)\text{ such that } \begin{vmatrix}
			a & c+1\\
			b & d+1
		\end{vmatrix}=0\}.$
	\end{center}
\end{enumerate}
For example, for any $a\in \ZY\cup\{-1\}$, we have $\underline{\ell}^{(a,0)}=\ell^{(0,-1)}$ and $\underline{\ell}^{(0,a)}=\ell^{(-1,0)}$.
\begin{rem}\thlabel{N2}
	{\rm 
	For any $(a,b)\in \ZZ^2\setminus\{(0,0)\}$, we know that $\underline{\ell}^{(a,b)}\cap \mathcal{M}$ is a ray.
	Hence, it has a minimum element with respect to the lexicographic ordering.
%	 prioritizing the first component.
	Let $(p,q)$ denote the minimum element of $\underline{\ell}^{(a,b)}\cap \mathcal{M}$.
	Then for any $(c,d)\in~\underline{\ell}^{(a,b)}\cap \mathcal{M}$, $(c-p,d-q)\in \ell_{(a,b)}$.
	Hence, by \thref{N1}, we can conclude that
	$$\underline{\ell}^{(a,b)}\cap \mathcal{M}=\bigg\{\bigg(p+\dfrac{ia}{\gcd(a,b)},q+\dfrac{ib}{\gcd(a,b)}\bigg)\mid i\in \ZZ\bigg\}.$$
	}
	\end{rem}
	
\begin{enumerate}\addtocounter{enumi}{7}
\item  For $(a,b)\in \mathbb{Z}^2\setminus\{(-1,-1)\}$, $\widetilde{\ell}_{(a,b)}$ : the line parallel to $\ell^{(a,b)}$ passing through $(0,0)$, i.e., 
	\begin{center}
	$\widetilde{\ell}_{(a,b)}=\{(c,d)\mid (a+1)d-(b+1)c=0\}
	=\{(c,d)\text{ such that } \begin{vmatrix}
			a+1 & c\\
			b+1 & d
	\end{vmatrix}=0\}.$\\
	\end{center}
\end{enumerate}
It follows from the definitions that 
$$
\ell_{(1,1)}=\ell_{(a,a)}=\ell^{(a,a)}=\underline{\ell}^{(a,a)}=\widetilde{\ell}_{(a,a)}, \text{ for any }a\in \mathbb{Z}\setminus \{0,-1\}.
$$
\noindent
%Next we draw $\ell_{(a,b)}, \underline{\ell}^{(a,b)},\ell^{(a,b)}, \widetilde{\ell}_{(a,b)}$, for an arbitrary element of $(a,b)\in \ZZ^2$. 
Next we draw the above lines (5-8), for some arbitrary $(a,b)\in\ZZ^2$.

\begin{figure}[h]
	\begin{center}
		\tikzset{every picture/.style={line width=0.75pt}}        
		
		\begin{tikzpicture}[x=1pt,y=1pt,yscale=-1.25,xscale=1.25]

		\tikzstyle{conefill} = [fill=blue!20, draw = black!40]
		\tikzstyle{conefill_gamma} = [fill=red!20, draw = black!70]
			
		\coordinate (e1) at (20,0);
		\coordinate (e2) at (0,-20);
			
		%%%%%%%%%%%%%%%%% P1 %%%%%%%%%%%%
						
			\coordinate (O) at (20,230);
			\coordinate (Oxmax) at ($(O)+5*(e1)$);
			\coordinate (Oymax) at ($(O)+5*(e2)$);
			\coordinate (Ozmax) at ($(O)+2*(e1)+5*(e2)$);
			\coordinate (Owmax) at ($(O)+2.5*(e1)+5*(e2)$);
			\coordinate (Oxmin) at ($(O)-2*(e1)$);
			\coordinate (Oymin) at ($(O)-2*(e2)$);
			\coordinate (Ozmin) at ($(O)-(e2)-(e1)$);
			\coordinate (Owmin) at ($(O)$);
			\coordinate (T1) at ($(O)-3*(e2)+2*(e1)$);
			\coordinate (T2) at ($(O)-4*(e2)+2*(e1)$);
			
			\draw [->,color=black!90] (Oxmin) -- (Oxmax) node[below] {$x$-axis};
			\draw [->,color=black!90] (Oymin) -- (Oymax) node[left] {$y$-axis};
			\draw [->,color=blue] (Ozmin) -- (Ozmax) node[left] {$\underline{\ell}^{(a,b)}$};
			\draw [->,color=red] (Owmin) -- (Owmax) node[right] {$\ell_{(a,b)}$};
			
			\foreach \x in {-1,0,...,4}
			\foreach \y in {-1,0,...,4}
			{
				\node[draw=teal!70,circle,inner sep=1.5pt,fill=teal!70] at ($(O)+\x*(e1)+\y*(e2)$) {};
			}
			\foreach \x in {0,...,4}
			{
				\node[draw=teal!70,circle,inner sep=1.5pt,fill=teal!70] at ($(O)+\x*(e1)-(e2)$) {};
			}
			\node[below left] at ($(O)-(e1)-(e2)$) {\tiny $(-1,-1)$};
			\node[below right] at ($(O)$) {\tiny $(0,0)$};
			\node[below right] at ($(O)+2*(e1)+4*(e2)$) {\tiny $(a,b)$};
			\node[draw=red,circle,inner sep=2pt,fill=teal!70] at ($(O)$) {};
			\node[draw=red,circle,inner sep=2pt,fill=teal!70] at ($(O)+(e1)+2*(e2)$) {};
			\node[draw=red,circle,inner sep=2pt,fill=teal!70] at ($(O)+2*(e1)+4*(e2)$) {};
			\node[draw=blue,circle,inner sep=2pt,fill=teal!70] at ($(O)-(e1)-(e2)$) {};
			\node[draw=blue,circle,inner sep=2pt,fill=teal!70] at ($(O)+(e2)$) {};
			\node[draw=blue,circle,inner sep=2pt,fill=teal!70] at ($(O)+3*(e2)+(e1)$) {};
			
			\draw[black] (T1) node {Pictorial description of $\ell_{(a,b)}$ and $\underline{\ell}^{(a,b)}$};
			
			%%%%%%%%%%%%%%%%% P_2 %%%%%%%%%%%%
			
			\coordinate (O) at (200,230);
			\coordinate (Oxmax) at ($(O)+5*(e1)$);
			\coordinate (Oymax) at ($(O)+5*(e2)$);
			\coordinate (Ozmax) at ($(O)+5*(e1)+3*(e2)$);
			\coordinate (Owmax) at ($(O)+5*(e1)+3.34*(e2)$);
			\coordinate (Oxmin) at ($(O)-2*(e1)$);
			\coordinate (Oymin) at ($(O)-2*(e2)$);
			\coordinate (Ozmin) at ($(O)-(e2)-(e1)$);
			\coordinate (Owmin) at ($(O)$);
			
			\coordinate (T1) at ($(O)-3*(e2)+2*(e1)$);
			\coordinate (T2) at ($(O)-4*(e2)+2*(e1)$);
			
			\draw [->,color=black!90] (Oxmin) -- (Oxmax) node[below] {$x$-axis};
			\draw [->,color=black!90] (Oymin) -- (Oymax) node[left] {$y$-axis};
			\draw [->,color=blue] (Ozmin) -- (Ozmax) node[right] {${\ell}^{(a,b)}$};
			\draw [->,color=red] (Owmin) -- (Owmax) node[above] {$\widetilde{\ell}_{(a,b)}$};
			\foreach \x in {-1,0,...,4}
			\foreach \y in {-1,0,...,4}
			{
				\node[draw=teal!70,circle,inner sep=1.5pt,fill=teal!70] at ($(O)+\x*(e1)+\y*(e2)$) {};
			}
			\foreach \x in {0,...,4}
			{
				\node[draw=teal!70,circle,inner sep=1.5pt,fill=teal!70] at ($(O)+\x*(e1)-(e2)$) {};
			}
			\node[below left] at ($(O)-(e1)-(e2)$) {\tiny $(-1,-1)$};
			\node[below left] at ($(O)$) {\tiny $(0,0)$};
			\node[below right] at ($(O)+2*(e1)+1*(e2)$) {\tiny $(a,b)$};
			\node[draw=red,circle,inner sep=2pt,fill=teal!70] at ($(O)$) {};
			\node[draw=red,circle,inner sep=2pt,fill=teal!70] at ($(O)+3*(e1)+2*(e2)$) {};
			\node[draw=blue,circle,inner sep=2pt,fill=teal!70] at ($(O)-(e1)-(e2)$) {};
			\node[draw=blue,circle,inner sep=2pt,fill=teal!70] at ($(O)+2*(e1)+(e2)$) {};
			\draw[black] (T1) node {Pictorial description of $\ell^{(a,b)}$ and $\widetilde{\ell}_{(a,b)}$};
		\end{tikzpicture}
	\end{center}
\end{figure}
\begin{rem}\thlabel{N3}
	{\rm
		Note that for any $(a,b), (c,d)\in \mathbb{Z}^2\setminus \{(0,0),(-1,-1)\}$,
		$$(a,b)\in \underline{\ell}^{(c,d)} \iff \ell^{(a,b)}=\underline{\ell}^{(c,d)} \iff
		\ell_{(c,d)}=\widetilde{\ell}_{(a,b)} 
		\iff (c,d)\in \widetilde{\ell}_{(a,b)}.$$
	}
\end{rem}
%\noindent
\begin{enumerate}\addtocounter{enumi}{8}
\item HDer($\kk[x,y]$) : the set of all homogeneous derivations in Der($\kk[x,y]$) with respect to the standard $\mathbb{Z}^2$-grading 

\item 
	Next we fix some notations for homogeneous derivations:
%	 of Lie(Aut($\mathbb{A}^2$)) :
\begin{equation*}
	\Delta := 
	x\dfrac{\partial}{\partial x}+y\dfrac{\partial}{\partial y}
\end{equation*}
%\item For any $(a,b)\in \Z^2$,
\begin{equation*}
	\partial_{(a,b)}:= x^ay^b\bigg((b+1)x\dfrac{\partial}{\partial x}-(a+1)y\dfrac{\partial}{\partial y}\bigg)\\
	,\text{ for } (a,b)\in \Z^2\\
\end{equation*}
%\begin{equation*}
%	\partial_{(m,-1)}:= -(m+1)x^m\dfrac{\partial}{\partial y},\text{ for } m\in \mathbb{Z}_{\geqslant 0}
%\end{equation*}
%\begin{equation*}
%	\partial_{(-1,n)}:=(n+1)y^n\dfrac{\partial}{\partial x},\text{ for } n\in \mathbb{Z}_{\geqslant 0}.
%\end{equation*}
%For convenience we define 
%Note that
%\begin{equation*}
%	\partial_{(-1,-1)}=0.
%\end{equation*}

\item	
For any 
$\mathcal{S}\subseteq \mathcal{M}$, Lie($\mathcal{S}$) : the subalgebra of Der($\kk[x,y]$) generated by the following set $\{\partial_{(a,b)}\mid (a,b)\in \mathcal{S}\}$

\item
For any $\mathcal{S}\subseteq \mathcal{M}$ and $\delta\in$ Der$(\kk[x,y])$, Lie($\mathcal{S}, \delta)$ : the subalgebra of Der($\kk[x,y]$) generated by the following set $\{\partial_{(a,b)}\mid (a,b)\in \mathcal{S}\}\cup \{\delta\}$

\item
 For any line $\ell$, Lie($\ell$) : the subalgebra of Der($\kk[x,y]$) generated by the following set $\{\partial_{(a,b)}\mid (a,b)\in \mathcal{M}\cap \ell\}$
\end{enumerate}
\begin{rem}\thlabel{N4}
	{\rm {Note that $\mathcal{M}\cap \ell$ can be finite, even empty also.
	For example, when \\ $m\in \ZZ$ then	
$\mathcal{M}\cap {\ell}_{(m,-1)}= \{(0,0),(m,-1)\}$ and $\mathcal{M}\cap \underline{\ell}^{(m,-1)}= \emptyset$.
Hence,  \begin{center}
	Lie(${\ell}_{(m,-1)}$) = $\kk\partial_{(0,0)}\oplus\kk\partial_{(m,-1)}$ and Lie($\underline{\ell}^{(m,-1)}$) = $0$.
\end{center}
}}
\end{rem}

\begin{enumerate}\addtocounter{enumi}{13}
\item For any line $\ell$ and $(c,d)\in \mathcal{M}$, Lie($\ell, (c,d)$) : the subalgebra of Der($\kk[x,y]$) generated by the following set $\{\partial_{(a,b)}\mid (a,b)\in \mathcal{M}\cap \ell\}\cup \{\partial_{(c,d)}\}$

\item For any line $\ell$ and $\delta\in$ Der$(\kk[x,y])$, Lie($\ell,\delta$) : the subalgebra of Der($\kk[x,y]$) generated by the following set $\{\partial_{(a,b)}\mid (a,b)\in \mathcal{M} \cap \ell\}\cup\{\delta\}$

\item For any line $\ell$ and a subalgebra $\mathfrak{h}$ of Der$(\kk[x,y])$, Lie($\ell,\mathfrak{h}$) : the subalgebra of Der($\kk[x,y]$) generated by the following set $\{\partial_{(a,b)}\mid (a,b)\in \mathcal{M} \cap \ell\}\cup\mathfrak{h}$

\item $\mathbb{T}$ : the standard maximal torus of Aut($\mathbb{A}^2$)

\item $\mathfrak{t}_2$ : the Abelian toral Lie subalgebra corresponding to $\mathbb{T}$, i.e., Lie($\mathbb{T}$)= Lie($\{(0,0)\},\Delta$)
%i.e.,
% = Lie($\mathbb{T}$) = Lie($(0,0),\Delta$) = $\kk x\dfrac{\partial}{\partial x}\oplus \kk y\dfrac{\partial}{\partial y} $, with the Lie bracket defined by \eqref{CR3}

 \item JONQ$^{\pm}(\mathbb{A}^2)$ : the triangular de Jonqu\'eres subgroups of Aut($\mathbb{A}^2$)
 
 \item $\mathfrak{j}^{\pm}=$ Lie(JONQ$^{\pm}(\mathbb{A}^2)$) : the triangular Lie subalgebras of Lie(Aut($\mathbb{A}^2$))
 
\item For $n\in \ZZ\cup\{-1\}$, $\mu\in \{0,1\}$ and $\pmb{\lambda}, \pmb{\alpha}, \pmb{\gamma}\in$ Map($\ZY,\kk$),  $\mathscr{G}(n,\mu,\pmb{\lambda},\pmb{\alpha},\pmb{\gamma})$ :
the Lie algebra on the vector space
$$\kk Y_1\oplus \kk Y_2\oplus\bigg(\overunderset{\infty}{i=0}{\oplus}\kk X_i\bigg)$$ {equipped with the following commutation relations:
\[	\begin{array}{llllll}
&[Y_1,X_0] = X_0,\, &[Y_2,X_0]=\mu X_0, &\\
&[Y_1,X_i] =\lambda_i X_i,\, &[Y_2,X_i]=\alpha_i X_i,\,
	&[X_0,X_i]=\gamma_i X_{i+n},\, \, \forall i\in \ZY, 	
	\end{array}
	\]
	where $\pmb{\lambda}=(\lambda_i)_i,\, \pmb{\alpha}=(\alpha_i)_i$, $\pmb{\gamma}=(\gamma_i)_i$ and
	all other commutators are zero.}
	In particular,  
%	\begin{center}
		$\mathfrak{j}_2^{\pm}= \mathscr{G}(-1,1,(2-i)_i,(-i)_i,(i-1)_i)$ (cf. \thref{IC0}(b)).
	
%    \end{center}
\end{enumerate}
{Henceforth, we will only define the non-zero commutators of a Lie algebra. }

\section{Preliminaries}\label{P}	

At first, we recall the definition of the derived length of a solvable Lie algebra.

\begin{Defn}\thlabel{D1}
{\rm 
	The {\it derived length} of a solvable Lie algebra $\mathfrak{L}$
	is the smallest\\ non-negative integer $r$, such that the $r$-th derived algebra $\mathfrak{L}^{(r)}$
	is zero, where
	\begin{center}
		
		$\mathfrak{L}^{(0)}:=\mathfrak{L}$ and $\mathfrak{L}^{(s+1)}:=[\mathfrak{L}^{(s)},\mathfrak{L}^{(s)}]$, for all $s\in \ZZ$.
	
	\end{center}
}	
\end{Defn}

The next result provides a uniform bound for the derived lengths of solvable Lie subalgebras of derivations \cite[Corollary~2(1)]{MP}.

\begin{prop}\thlabel{ThDL}
	Let $\mathfrak{L}$ be a solvable subalgebra of
	Der$(\kk[x_1,\dots,x_n])$, for some \\ $n\in \ZY$. Then the derived length of $\mathfrak{L}$ is at most 2n.	
\end{prop}

Recall that
%\begin{center}
 Der($\kk[x,y]$)~=~Vec($\mathbb{A}^2$) consists of all polynomial vector fields over $\mathbb{A}^2$
% \end{center}
 and 
 \begin{center}
 Lie(Aut($\mathbb{A}^2$)) = Vec$^c(\mathbb{A}^2)$ consists of polynomial vector fields with constant divergence,
 \end{center}
 where divergence of a vector field $\delta=f(x,y)\partial/\partial x + g(x,y)\partial/\partial y$ is defined as follows:
 \begin{center}
 	div($\delta$) = ${\partial f}/{\partial x}+{\partial g}/{\partial y}$.
 \end{center}
 The Lie subalgebra Vec$^0(\mathbb{A}^2)\subset$ Vec$^c(\mathbb{A}^2)$ consisting of vector fields over $\mathbb{A}^2$ with zero divergence coincides with Lie($\mathfrak{U}_2$), where $\mathfrak{U}_2$ is the subgroup of Aut($\mathbb{A}^2$) consisting of all automorphisms of $\mathbb{A}^2$ with Jacobian one  \cite[Proposition 15.7.2]{FK}.
 Moreover, in \\ \cite[Section~6.2]{KZ1} there is a decomposition as follows:
 \begin{equation}\label{ED}
 	\text{Vec}^c(\mathbb{A}^2)=
% 	\text{Vec}^0 (\mathbb{A}^2) \oplus \kk(x\partial/\partial x + y\partial/\partial y)=
\text{Vec}^0 (\mathbb{A}^2) \oplus \kk\Delta, \text{ where }\Delta =x\partial/\partial x + y\partial/\partial y.
 \end{equation}
 
% Now we discuss the bigraded structure of Lie(Aut($\mathbb{A}^2$)).
 The standard $\mathbb{Z}^2$-grading on the polynomial ring $\kk[x, y]$ induces a $\mathbb{Z}^2$-grading on the Lie algebras Der($\kk[x,y]$), Lie(Aut($\mathbb{A}^2$)) and
Lie($\mathfrak{U}_2$). 
Hence, Der($\kk[x,y]$), Lie(Aut($\mathbb{A}^2$)) and Lie($\mathfrak{U}_2$) are bigraded Lie algebras and
for any $i, j \in \ZZ$, we have
\begin{equation*}
	\text{bideg}(x^iy^j\partial/\partial x) = (i-1, j) \text{ and }
	\text{bideg}(x^iy^j\partial/\partial y) = (i, j-1).
\end{equation*}
Thus, the bidegree of homogeneous derivations are of the form $(m,n)\in \mathbb{Z}^2$, such that $(m,n)\in \mathbb{Z}_{\geqslant 0}^2\cup\{(m,-1)\mid m\in \mathbb{Z}_{\geqslant 0}\}\cup\{(-1,n)\mid n\in \mathbb{Z}_{\geqslant 0}\}=\mathcal{M}.$
If $\delta\in$ HDer($\kk[x,y]$) is 
of bidegree $(a,b)$ and $x^iy^j\notin$ ker($\delta$) then bideg($\delta(x^iy^j)$) = $(a+i,b+j)$.
For non-commuting $\delta,\partial\in$HDer($\kk[x,y]$), we have bidegree [$\delta,\partial$] = bidegree($\delta$)$+$bidegree($\partial$). 

The homogeneous elements of Vec$^0(\mathbb{A}^2)$ with bidegree $(a,b)\in \mathcal{M}$ are of the form $\partial_{(a,b)}$ (cf. Section~\ref{N}, notation~10) up to multiplication by a scalar.
Next by \eqref{ED}, we can conclude that the homogeneous elements of Lie(Aut($\mathbb{A}^2$)) with bidegree $(a,b)\in \mathcal{M}\setminus\{(0,0)\}$ are the same.
Only for bidegree $(0,0)$, the homogeneous elements are of the form 
\begin{center}

$\lambda\partial_{(0,0)}+\mu\Delta=(\lambda+\mu)x\dfrac{\partial}{\partial x} + (\mu-\lambda)y\dfrac{\partial}{\partial y}$, for  $\lambda,\mu\in \kk$ both not zero simultaneously.

\end{center}
These homogeneous derivations are endowed with the following commutation relations:

\begin{equation}\label{CR1}
	[\partial_{(a,b)},\partial_{(c,d)}]	=
	\begin{vmatrix}
		c+1 & a+1\\
		d+1 & b+1
	\end{vmatrix}
	\partial_{(a+c,b+d)}, \text{ for }(a,b),(c,d) \in \mathcal{M}
\end{equation}

%\begin{equation}\label{CR2}
%	[\partial_{(a,b)},\partial_{(c,d)}]	= 0, \text{ for }(a,b),(c,d)\in \mathcal{M} \text{ and } (a+c,b+d)\notin \mathcal{M} 	
%\end{equation}

\begin{equation}\label{CR3}
	[\Delta,\partial_{(c,d)}]= (c+d)\partial_{(c,d)},\text{ for } (c,d)\in\mathcal{M}. 	
\end{equation}
Note that every homogeneous locally nilpotent derivation (cf. \cite[1.1.7]{GFB}) of $\kk[x,y]$ is of the form $\partial_{(a,b)}$, for some $(a,b)\in \mathcal{M}\setminus \ZZ^2$ and hence is contained in Lie($\mathfrak{U}_2$).
% Vec$^0(\mathbb{A}^2)\subset$ Lie(Aut($\mathbb{A}^2$)). 
%Recall that the commutation relations of these homogeneous derivations are already defined in  Section~\ref{N} equations~\eqref{CR1}, \eqref{CR2} and \eqref{CR3}.

Next, we write down an observation in the form of a lemma.
It is used extensively throughout.

\begin{lem}\thlabel{P3}
	For any $(a,b), (c,d)\in \mathcal{M}$,
	the following statements hold:
	\begin{enumerate}[\rm(a)]
	\item if $(a,b)\in \ell^{(c,d)}$ then $[\partial_{(a,b)}, \partial_{(c,d)}]=0$.
	\item if $(a,b)\notin \ell^{(c,d)}$ then 
	$\partial_{(a+c,b+d)}\in$ Lie$(\{(a,b), (c,d)\})$.
   \end{enumerate}
\end{lem}
\begin{proof}
	{\bf (a) : }If $(a,b)\in \ell^{(c,d)}$ then $
	\begin{vmatrix}
		c+1 & 	a+1\\
		d+1 & b+1
	\end{vmatrix}= 0$.
	Therefore, by commutation relation~\eqref{CR1}, we have $[\partial_{(a,b)}, \partial_{(c,d)}]=0$.
	
	{\bf (b) : }
	If $(a,b)\notin \ell^{(c,d)}$ then $
	\begin{vmatrix}
		c+1 & 	a+1\\
		d+1 & b+1
	\end{vmatrix}\neq 0$.
	Therefore, either $(a+c,b+d)\in \mathcal{M}$ or $\{(a,b),(c,d)\}=\{(0,-1), (-1,0)\}$.
	Now, 
	\begin{center}
	$[\partial_{(-1,0)}, \partial_{(0,-1)}]=\partial_{(-1,-1)}$ and
	$\partial_{(-1,-1)}=0$.
\end{center}
% when $(a+c,b+d)\in \mathcal{M}$, by commutation relation~\eqref{CR1}, we have
%	\begin{equation*}
%		[\partial_{(a,b)},\partial_{(c,d)}]	=
%		\begin{vmatrix}
%			c+1 & a+1\\
%			d+1 & b+1
%		\end{vmatrix}
%		\partial_{(a+c,b+d)}.
%	\end{equation*}
	Hence, by  commutation relation~\eqref{CR1}, we have $\partial_{(a+c,b+d)}\in$ Lie($\{(a,b), (c,d)\}$).
%	 in both cases.
\end{proof}

Now, we recall the definition of some subalgebras of Lie(Aut($\mathbb{A}^2$)).
First, the Abelian toral Lie subalgebra
\begin{center}
$\mathfrak{t}_2$
= Lie($\mathbb{T}$) = Lie($(0,0),\Delta$) = $\kk \partial_{(0,0)}\oplus \kk \Delta $ = $\kk x\dfrac{\partial}{\partial x}\oplus \kk y\dfrac{\partial}{\partial y} $
\end{center}
and next, the triangular Lie subalgebras 
\begin{center}
	$\mathfrak{j}_2^{+}$ = Lie(JONQ$^{+}(\mathbb{A}^2)$) =
	$\kk[y]\dfrac{\partial}{\partial x}\oplus\kk x\dfrac{\partial}{\partial x}\oplus \kk y\dfrac{\partial}{\partial y}\oplus \kk\dfrac{\partial}{\partial y}$
%	 Lie($\mathfrak{t}_2,\underline{\ell}^{(0,-1)}, (0,-1)$) 
=	Lie($\mathfrak{t}_2,{\ell}^{(-1,0)}, (0,-1)$)
\end{center}
and
\begin{center}
	$\mathfrak{j}_2^{-}$ = Lie(JONQ$^{-}(\mathbb{A}^2)$) =
	$\kk[x]\dfrac{\partial}{\partial y}\oplus\kk x\dfrac{\partial}{\partial x}\oplus \kk y\dfrac{\partial}{\partial y}\oplus \kk\dfrac{\partial}{\partial x}$ =
	Lie($\mathfrak{t}_2,{\ell}^{(0,-1)}, (-1,0)$).
\end{center}
We already know that $\mathfrak{j}_2^{+}$ and $\mathfrak{j}_2^{-}$ are Borel subalgebras of Lie(Aut($\mathbb{A}^2$)) of derived length~$3$ \cite[Proposition~4.9]{AZ3} and they are also maximal regular solvable subalgebras, with respect to inclusion.
Our next step will be in this direction.
Precisely,  we will show that the set of all regular Borel subalgebras coincides with the set of all  maximal solvable subalgebras, in this case (\thref{RBS4}).
	
We begin by stating an easy lemma. For a similar proof one can look at \cite[Lemma~4.8]{AZ3}.

\begin{lem}\thlabel{RBS1}
	Let $V$ be a vector space and $L$ be a linear operator on $V.$
	Suppose $U$ is an invariant subspace of $V$ under $L$.
	
	Moreover, let 
	$u=\sum_{i=1}^{n}u_i\in U,$ for some $u_1,\dots,u_n\in V$ 	
	with $Lu_i=\lambda_iu_i$, for all $i$ such that $\lambda_i$'s are different from each other.
	
	Then $u_1,\dots,u_n\in U$.
\end{lem}
The next lemma gives a criterion for a subalgebra of Der($\kk[x,y]$) to be regular.

\begin{lem}\thlabel{RBS2}
	Every subalgebra of Der$(\kk[x,y])$ containing $\mathfrak{t}_2$ is regular.
\end{lem}
\begin{proof}
	Let $\mathfrak{L}$ be a subalgebra of Der$(\kk[x,y])$ containing $\mathfrak{t}_2$ and $\delta$ be an arbitrary element of $\mathfrak{L}$.
	Suppose 
	\begin{center}
		$\delta=\underset{(a,b)\in \mathcal{M}}{\sum}\delta_{(a,b)}$ is the weight decomposition of $\delta$,
	\end{center}
	where $\delta_{(a,b)}$ is the homogeneous component of bidegree $(a,b)$ (cf. \cite[Proposition~3.8]{GFB}).
	Now, we are going to show that $\delta_{(a,b)}\in \mathfrak{L}$, for all $(a,b)$.
	
	We consider two elements
	\begin{center}
		$A:=$ ad$\bigg(x\dfrac{\partial}{\partial x}\bigg)$, $B:=$ ad$\bigg(y\dfrac{\partial}{\partial y}\bigg)\in $ End(Der($\kk[x,y]$)).	
	\end{center}  
	Then $\delta_{(a,b)}$ is an eigenvector of $A$ (resp. of $B$) with eigenvalue $a$ (resp. $b$).
	Next $\mathfrak{t}_2\subseteq \mathfrak{L}$; hence, $\mathfrak{L}$ is invariant under both $A$ and $B$.
	Therefore, $A\delta=\underset{(a,b)}{\sum}a\delta_{(a,b)}\in \mathfrak{L}$.
	Now the \\$A$-homogeneous components of $\delta$ are of the form $\delta_a=\underset{b}{\sum}\delta_{(a,b)}$ and hence,  by \thref{RBS1}, we have $\delta_a\in \mathfrak{L}$, for all $a$.
	After that, $B(\delta_a)=\underset{b}{\sum}b\delta_{(a,b)}$ and the $B$-homogeneous components of $\delta_a$ are $\delta_{(a,b)}$.
	Therefore, again by \thref{RBS1}, we can conclude that $\delta_{(a,b)}\in \mathfrak{L}$, for all $(a,b)$.
	Thus, $\mathfrak{L}$ contains all the homogeneous component of every element of itself.
	Hence, $\mathfrak{L}$ is regular.
\end{proof}
Now we discuss two results concerning maximal regular solvable subalgebras.
% of Lie(Aut($\mathbb{A}^2$)).
\begin{lem}\thlabel{RBS3}
	Every maximal regular solvable subalgebra of Lie(Aut($\mathbb{A}^2$)) contains $\mathfrak{t}_2$.
\end{lem}
\begin{proof}
	Let $\mathfrak{L}$ be a maximal regular solvable subalgebra of Lie(Aut($\mathbb{A}^2$)).
	Consider the\\ subalgebra  Lie($\mathfrak{t}_2,\mathfrak{L}$) containing $\mathfrak{L}$.
	Then by \thref{RBS2}, Lie($\mathfrak{t}_2,\mathfrak{L}$) is regular.
	Note that $$\mathfrak{t}_2=\kk\partial_{(0,0)}\oplus\kk\Delta.$$
	Now, if $\delta_{(a,b)}$ is a homogeneous derivation of Lie(Aut($\mathbb{A}^2$)) of bidegree $(a,b)\in \mathcal{M}$, then by  commutation relation \eqref{CR1} (resp. \eqref{CR3}), it follows that $\delta_{(a,b)}$ is an eigenvector of ad($\partial_{(0,0)}$) (resp. ad($\Delta$)) with eigenvalue $a-b$ (resp. $a+b$).
	Hence, \begin{center}
		Lie($\mathfrak{t}_2,\mathfrak{L}$)$^{(1)}\subseteq \mathfrak{L}$.	
	\end{center} 
	Therefore,
	Lie($\mathfrak{t}_2,\mathfrak{L}$) is a regular solvable subalgebra.
	Hence, Lie($\mathfrak{t}_2,\mathfrak{L}) = \mathfrak{L}$.
\end{proof}
\begin{lem}\thlabel{RBS4}
	Every maximal regular solvable subalgebra of Lie(Aut($\mathbb{A}^2$)) is a Borel\\ subalgebra of Lie(Aut($\mathbb{A}^2$)). 	
\end{lem}
\begin{proof}
	Let $\mathfrak{L}$ be a maximal regular solvable subalgebra of Lie(Aut($\mathbb{A}^2$)).
	Then by\\ \cite[Corollary~2.3]{AZ3}, there exists a Borel subalgebra of Lie(Aut($\mathbb{A}^2$)), say $\mathfrak{B}$, containing~$\mathfrak{L}$.
	Now by \thref{RBS3}, we have $\mathfrak{t}_2\subseteq \mathfrak{L}\subseteq \mathfrak{B}$. 
	Hence, by \thref{RBS2}, it follows that $\mathfrak{B}$ is regular.
	Therefore, $\mathfrak{L}=\mathfrak{B}$ and the result follows.
\end{proof}
\begin{rem}\thlabel{RBS5}
	{\rm 
	It clearly follows from the above statement that
	the set of all regular Borel subalgebras of Lie(Aut($\mathbb{A}^2$)) coincides with the set of all maximal regular solvable subalgebras of Lie(Aut($\mathbb{A}^2$)).
	}
\end{rem}

Let $\mathfrak{B}$ be an arbitrary regular Borel subalgebra of Lie(Aut($\mathbb{A}^2$)).
Then by \thref{RBS5} and \thref{RBS3}, it follows that $\mathfrak{B}$ is of the following form:
\begin{equation}\label{PE0}
	\mathfrak{B}= \text{Lie}(\mathfrak{t}_2,\mathcal{S}) = \text{Lie}(\Delta,\{(0,0)\}\cup\mathcal{S}), \text{ for some } \mathcal{S}\subseteq \mathcal{M}\setminus\{(0,0)\}.
\end{equation}
Now, $\mathfrak{t}_2\subsetneq \mathfrak{j}^{+},$ a Borel subalgebra of Lie(Aut($\mathbb{A}^2$)).
Hence, $\mathfrak{t}_2$ is not a regular Borel subalgebra of Lie(Aut($\mathbb{A}^2$)).
Therefore,
\begin{equation}\label{PE01}
	\mathcal{S}\neq \emptyset.
\end{equation}
Next, by commutation relations \eqref{CR1} and \eqref{CR3} respectively, we have
\begin{center}
	$[\partial_{(0,0)}, \partial_{(c,d)}]=(c-d)\partial_{(c,d)}$
	and 
	$[\Delta, \partial_{(c,d)}]=(c+d)\partial_{(c,d)}$, for all $(c,d)\in \mathcal{M}$.
\end{center}
Hence,
\begin{equation}\label{PE1}
	\mathfrak{B}^{(1)}= \text{Lie}(\mathcal{S}).
\end{equation}
So in the next two sections of this paper, we study {\it regular Abelian subalgebras}, i.e., Abelian subalgebras generated by
homogeneous derivations with respect to the standard $\mathbb{Z}^2$-grading  and  {\it regular solvable subalgebras}, i.e., solvable subalgebras generated by
homogeneous derivations with respect to the standard $\mathbb{Z}^2$-grading.
% of the form Lie($\mathcal{S}$), in Section~\ref{RAS} and Section~\ref{RSS} respectively. 
%for some $\mathcal{S}\subset \mathcal{M}\setminus\{(0,0)\}$.

\section{Regular Abelian subalgebras}\label{RAS}
%
%%Following the definition of regular Borel subalgebra (\thref{DF1}) 
%We now introduce regular Abelian subalgebra.
%
%\begin{Defn}
%{\rm 
%	An Abelian subalgebra of Lie(Aut($\mathbb{A}^2$)) is called  {\it regular} if it is generated by
%    homogeneous derivations with respect to the standard $\mathbb{Z}^2$-grading.
%}	
%\end{Defn}
%\noindent
In this section, we discuss regular Abelian subalgebras and
maximal regular Abelian subalgebras of Lie(Aut($\mathbb{A}^2$)).
In Propositions~\ref{RAS2} and \ref{RAS3}, we show that
the maximal regular Abelian subalgebras of Lie(Aut($\mathbb{A}^2$)) are also maximal Abelian subalgebras. 
At first, we describe all the regular Abelian subalgebras containing $\Delta$.
\begin{lem}\thlabel{RAS0}
	Let $\mathfrak{L}$ be a regular subalgebra of Lie(Aut($\mathbb{A}^2$)) containing $\Delta$.
	Then $\mathfrak{L}$ is Abelian if and only if either $\mathfrak{L}$ is equal to  Lie$(\Delta, \partial_{(1,-1)})$ or Lie$(\Delta, \partial_{(-1,1)})$ or $ \mathfrak{t}_2$.
\end{lem}
\begin{proof}
	By commutation relation~\eqref{CR3}, it follows that $[\Delta,\partial_{(c,d)}]=0$, if and only if \\$(c,d)\in \{(0,0),(1,-1),(-1,1)\}$.
	Now, by commutation relation~\eqref{CR1}, we have
	\begin{center}
	$[\partial_{(0,0)},\partial_{(1,-1)}]=2\partial_{(1,-1)}$,
	$[\partial_{(-1,1)},\partial_{(0,0)}]=2\partial_{(-1,1)}$
	and $[\partial_{(-1,1)},\partial_{(1,-1)}]=4\partial_{(0,0)}$.
	\end{center}
	Hence, the result follows.
\end{proof}

Next, we consider the regular subalgebras not containing $\Delta$ and give a criterion for commutativity.
\begin{lem}\thlabel{RAS1}
	For any $\mathcal{S}\subseteq \mathcal{M}$,
	Lie$(\mathcal{S})$ is Abelian if and only if either  there exists a line passing through $(-1,-1)$ containing $\mathcal{S}$ or $\mathcal{S}=\{(0,-1),(-1,0)\}$.
\end{lem}
\begin{proof}
	By commutation relation~\eqref{CR1} and $\partial_{(-1,-1)}=0$, it follows that Lie($\{(0,-1),(-1,0)\}$) is Abelian.
	
	Next, suppose there exists a line containing $|{(-1,-1)}\cup\mathcal{S}$.
	Let $(a,b),(c,d)\in\mathcal{S}\subseteq\mathcal{M}$.
	Then $(a,b)\in \ell^{(c,d)}$. 
	Hence, by \thref{P3}(a), we have $[\partial_{(a,b)}, \partial_{(c,d)}]=0$.
Therefore, Lie($\mathcal{S}$) is Abelian.
	
	\smallskip
	Now we prove the converse.
	Suppose $\mathcal{S}\subseteq \mathcal{M}$ and Lie($\mathcal{S}$) is Abelian.
	We now divide the proof in two cases.
	
	\noindent
	{\bf Case I : } Suppose $\mathcal{S}=\{(a,b),(c,d)\}$.\\
	\noindent
%	Let  be two distinct arbitrary points of $\mathcal{S}$.
	Then $[\partial_{(a,b)}, \partial_{(c,d)}]=0$.
	Hence, by commutation relation~\eqref{CR1}, either
	\begin{center} 
	$\partial_{(a+c,b+d)}=0$, i.e.,	$\{(a,b),(c,d)\}= \{(0,-1),(-1,0)\}$ 	
		\end{center}	
	or 
	\begin{center}
		$\begin{vmatrix}
			c+1 & a+1\\
			d+1 & b+1
		\end{vmatrix}=0$, i.e., $(a,b)\in \ell^{(c,d)}$.
	\end{center}
	\noindent
	{\bf Case II : } Suppose $\mid\mathcal{S}\mid\geqslant 3$.\\
	\noindent
	Note that by commutation relation~\eqref{CR1}, we have
	\begin{center}
		$[\partial_{(0,-1)}, \partial_{(a,b)}]=-(b+1)\partial_{(a,b-1)}$ and 
		$[\partial_{(-1,0)}, \partial_{(a,b)}]=(a+1)\partial_{(a-1,b)}$, for any $(a,b)\in \mathcal{M}$.
	\end{center}
	Therefore, for any $(a,b)\in \mathcal{M}\setminus\{(0,-1),(-1,0)\}$, Lie($\{(0,-1),(-1,0),(a,b)\}$) is not Abelian.
	Hence, it follows that $\{(0,-1),(-1,0)\}\not\subseteq\mathcal{S}$.
	Now, by Case I, for any $(a,b),(c,d)\in \mathcal{S}$, we have $(a,b)\in \ell^{(c,d)}.$
	Therefore, $\mathcal{S}\subseteq\ell^{(a,b)}$, for any $(a,b)\in \mathcal{S}$.
	\end{proof}
%	Now we are going to discuss regular Abelian subalgebras containing $\Delta$.
	\begin{rem}\thlabel{RAS4}
	{\rm
		By \thref{RAS0} and \thref{RAS1}, it follows that   Lie$(\Delta, \partial_{(1,-1)})$,
		Lie$(\Delta, \partial_{(-1,1)})$, $ \mathfrak{t}_2$ (= Lie$(\Delta, \{(0,0)\})$), Lie$(\{(0,-1),(-1,0)\})$ and Lie($\ell^{(a,b)}$), for any $(a,b)\in \mathcal{M}$  are the only maximal regular Abelian subalgebras of 
		Lie(Aut($\mathbb{A}^2$)).
	}	
\end{rem}

	The next two propositions (\ref{RAS2} and \ref{RAS3}) show that the maximal regular Abelian subalgebras of Lie(Aut($\mathbb{A}^2$)) are also maximal Abelian subalgebras.
	
	\begin{prop}\thlabel{RAS2}
		The subalgebras Lie$(\{(0,-1),(-1,0)\})$, Lie$(\Delta, \partial_{(1,-1)})$,
		Lie$(\Delta, \partial_{(-1,1)})$ and $\mathfrak{t}_2$ are maximal  Abelian subalgebras of 
		Lie(Aut($\mathbb{A}^2$)).	
	\end{prop}
	\begin{proof}Let us first prove the case of Lie($\{(0,-1), (-1,0)\}$).
		
		Let $\mathfrak{L}$ denote the subalgebra Lie($\{(0,-1), (-1,0)\}$).
		By \thref{RAS1}, we know that $\mathfrak{L}$ is Abelian.
		Now we are going to show that $\mathfrak{L}$ is maximal among all the Abelian subalgebras of Lie(Aut($\mathbb{A}^2$)). 
		
		Let $\delta\in$ Lie(Aut($\mathbb{A}^2$)) be such that Lie$(\{(0,-1),(-1,0)\},\delta)$ is Abelian. 
		Then
		\begin{center}
			$[\delta,\partial_{(0,-1)}]=[\delta,\partial_{(-1,0)}]=[\partial_{(0,-1)},\partial_{(0,-1)}]=0$.
		\end{center}
		Now Lie(Aut($\mathbb{A}^2$)) is a bigraded Lie algebra with
		\begin{center}
			bideg($x^iy^j\partial/\partial x$)=($i-1,j$) and bideg($x^iy^j\partial/\partial y$)=($i,j-1$), for any $i,j\in \ZZ$.
		\end{center}
		Therefore,
		by commutation relations~\eqref{CR1} and \eqref{CR3}, we have 
		$$\delta\in \kk\partial_{(0,-1)}\oplus\kk\partial_{(-1,0)}.$$
%		\begin{center} $\delta=\lambda\partial_{(0,-1)}+\mu\partial_{(-1,0)}$, for some $\lambda,\mu \in \kk$.
%		\end{center}
		Hence, $\mathfrak{L}$ is a maximal  Abelian subalgebra of 
		Lie(Aut($\mathbb{A}^2$)).
		
		The rest of the cases follow similarly, with the help of \thref{RAS0} instead of \thref{RAS1}.	
	\end{proof}
	It is clear that all the Abelian Lie subalgebras mentioned in \thref{RAS2} are isomorphic.
	
	Recall that for any $(a,b)\in \mathcal{M}$, Lie$(\ell^{(a,b)})$ is the subalgebra generated by the set $\{\partial_{(c,d)}\mid (c,d)\in \mathcal{M}\cap \ell^{(a,b)}\}$ and  the cardinality of $\mathcal{M}\cap \ell^{(a,b)}$ is $\aleph_0$.
	\begin{prop}\thlabel{RAS3}
		For any $(a,b)\in \mathcal{M}$, Lie$(\ell^{(a,b)})$ is a maximal Abelian subalgebra of Lie(Aut($\mathbb{A}^2$)).	
\end{prop}
\begin{proof}
	Let $(a,b)\in \mathcal{M}$.
	By \thref{RAS1} we know that Lie$(\ell^{(a,b)})$ is Abelian.
	Now, we are going to show that Lie$(\ell^{(a,b)})$ is maximal among all the Abelian subalgebras of Lie(Aut($\mathbb{A}^2$)).  
	
	Let $\delta\in $ Lie(Aut($\mathbb{A}^2$)) be
	such that Lie$(\ell^{(a,b)},\delta)$ is Abelian. 
	Then
	\begin{center}
		$[\delta, \partial_{(c,d)}]=0$,  for all $(c,d)\in \ell^{(a,b)}\cap \mathcal{M}$.
	\end{center}
	Recall that Lie(Aut($\mathbb{A}^2$)) is a bigraded Lie algebra with
	\begin{center}
		bideg($x^iy^j\partial/\partial x$)=($i-1,j$) and bideg($x^iy^j\partial/\partial y$)=($i,j-1$), for any $i,j\in \ZZ$.
	\end{center}
	Therefore, there exists $\mathcal{S}\subseteq \mathcal{M}$ such that
	\begin{center}
		$\delta=\underset{(r,s)\in \mathcal{S}}{\sum}\lambda_{(r,s)}\partial_{(r,s)}$, $[\partial_{(c,d)},\partial_{(r,s)}]=0, \, \forall \,(c,d)\in \ell^{(a,b)}\cap \mathcal{M}$ and $\lambda_{(r,s)}\in \kk^{*},\, \forall (r,s)\in \mathcal{S}$.
	\end{center} 
	Hence, for each $(r,s)\in \mathcal{S}$, the subalgebra Lie($\ell^{(a,b)},\{(r,s)\}$) is Abelian.
	Therefore, by \\
	\thref{RAS1}, we have 
	$$\mathcal{S}\subseteq \ell^{(a,b)}.$$
	Hence, Lie$(\ell^{(a,b)})$ is a maximal Abelian subalgebra of Lie(Aut($\mathbb{A}^2$)).
	\end{proof}
	It is evident that for any $(a,b),(c,d)\in \mathcal{M}$, Lie($\ell^{(a,b)}$) is isomorphic to Lie($\ell^{(c,d)}$), because both have countable basis of cardinality $\aleph_0$.
 
 It is also clear that Lie$(\{(0,-1),(-1,0)\})$ and Lie($\ell^{(a,b)}$) are never isomorphic.
	
\section{Regular solvable subalgebras}\label{RSS}
%Following the definition of regular Borel subalgebra (\thref{DF1}) 
%We now introduce regular solvable subalgebra of Lie(Aut($\mathbb{A}^2$)).
%\begin{Defn}
%{\rm	
%	A solvable subalgebra of Lie(Aut($\mathbb{A}^2$)) is called  {\it regular} if it is generated by
%	homogeneous derivations with respect to the standard $\mathbb{Z}^2$-grading.
%}		
%\end{Defn}
In this section, our goal is to provide a criterion for solvability of regular solvable subalgebras of Lie(Aut($\mathbb{A}^2$)) (\thref{GC4}).
First, we consider regular subalgebras generated by two homogeneous derivations in Subsection~\ref{RSST} and provide a necessary and sufficient condition for solvability (\thref{RSST0} and \thref{Th1}).
Next, in Subsection~\ref{GC}, we consider the general case.

%provide a criteria of solvability for regular subalgebras of Lie(Aut($\mathbb{A}^2$)) (\thref{GC2})
We now prove an easy lemma.
%some easy results (\ref{P1}, \ref{P11} and \ref{P2}) in this regard.
%Now we prove an easy lemma.
\begin{lem}\thlabel{P1}
	Let $(a,b),(c,d)\in \Z^2$ be such that $(a,b)\neq (0,0)$ and $(c,d)\neq (-1,-1)$.
	If $(c,d)\in \underline{\ell}^{(a,b)}$ then $(ra+c,rb+d)\in \ell^{(c,d)}=\underline{\ell}^{(a,b)}$, for all $r\in \ZY$.
	
	Conversely, if $(ra+c,rb+d)\in \ell^{(c,d)}$, for some $r\in \ZY$ then $(c,d)\in \underline{\ell}^{(a,b)}.$
\end{lem}
\begin{proof}
	Note that for any $r\in \ZY$, we have
	$\begin{array}{ccccc}
		r\begin{vmatrix}
			a & c+1\\
			b & d+1\\ 
		\end{vmatrix}
		&=&\begin{vmatrix}
			ra+c+1 & c+1\\
			rb+d+1 & d+1\\ 
		\end{vmatrix}.
	\end{array}$
	
	\noindent
	Now, if $(c,d)\in \underline{\ell}^{(a,b)}$ then
%	\begin{center}
	$\ell^{(c,d)}=\underline{\ell}^{(a,b)}$ and $\begin{vmatrix}
		a & c+1\\
		b & d+1\\ 
	\end{vmatrix}= 0.$
%	\end{center}
	Hence, $\begin{vmatrix}
		ra+c+1 & c+1\\
		rb+d+1 & d+1\\ 
	\end{vmatrix}~=~0,$ for all $r\in \ZY$.
	Therefore,  $(ra+c,rb+d)\in \ell^{(c,d)}=\underline{\ell}^{(a,b)},\, \forall r\in \ZY$.
	
	Next writing the arguments in reverse order, one gets the converse.	
\end{proof}
%The next corollary follows similarly.
%\begin{cor}\thlabel{P11}
%	Let $(a,b),(c,d)\in \Z^2$ be such that $(a,b)\neq (0,0)$ and $(c,d)\neq (-1,-1)$.
%	If $(c,d)\notin \underline{\ell}^{(a,b)}$ then $(ra+c,rb+d)\notin \ell^{(c,d)}$, for all $r\in \ZY$.
%	
%	Conversely, if $(ra+c,rb+d)\notin \ell^{(c,d)}$, for some $r\in \ZY$ then we have $(c,d)\notin \underline{\ell}^{(a,b)}.$	
%\end{cor}
Next we note down an observation in the form of a lemma.
\begin{lem}\thlabel{P2}
	Suppose $(a,b),(c,d)\in \mathcal{M}$ is such that $\{(a,b),(c,d)\}\neq \{(0,-1), (-1,0)\}$.\\
	If $(c,d)\in \underline{\ell}^{(a,b)}$ and $(a,b)\in \underline{\ell}^{(c,d)}$ then $(c,d)\in \ell^{(a,b)}$. 
\end{lem}
\begin{proof}
	Suppose $(c,d)\in \underline{\ell}^{(a,b)}$ and $(a,b)\in \underline{\ell}^{(c,d)}$.
	Then by \thref{N3}, we have
	$$(c,d)\in  \underline{\ell}^{(a,b)}\cap\widetilde{\ell}_{(a,b)}.$$
	Now $\{(a,b), (c,d)\}\neq \{(0,-1), (-1,0)\}$; hence, we have
	$\ell^{(c,d)}=\underline{\ell}^{(a,b)}=\widetilde{\ell}_{(a,b)}=\ell^{(a,b)}.$
	Therefore, $(c,d)\in \ell^{(a,b)}$.	
\end{proof}
\subsection{Solvability of two generated regular subalgebras}\label{RSST}

In this subsection, our goal is to provide a necessary and sufficient condition for solvability of regular subalgebras generated by two homogeneous derivations (\thref{RSST0} and \thref{Th1}).
%We first prove a preparatory lemma 
%in this regard.
%
%Consider $(a,b)\in \mathcal{M}$.
%Then by the commutation relation~\eqref{CR1} either Lie($\{(a,b), (0,0)\}$) is Abelian (when $a=b$) or it is isomorphic to the non-abelian Lie algebra of dimension~$2$ (when $a\neq b$).
%Therefore, Lie($\{(a,b), (0,0)\}$) is always solvable.
%Similarly, Lie($\{(a,b), \Delta\}$) is always solvable.
%In fact, it is Abelian if $a+b=0$ and non-abelian Lie algebra of dimension~$2$, otherwise.
At first, we consider the case when at least one homogeneous derivation has bidegree~$(0,0)$.
\begin{lem}\thlabel{RSST0}
For any $(a,b)\in \mathcal{M}$, Lie$(\{(a,b),(0,0)\})$ and Lie$(\{(a,b)\},\Delta)$	are solvable.
Moreover, they are either Abelian or solvable of derived length~$2$.
%if $a\neq b$ then Lie$(\{(a,b),(0,0)\})$ is non-Abelian and if $a+b\neq 0$ then Lie$(\{(a,b)\},\Delta)$ is non-Abelian.
%In the non-Abelian case, both Lie$(\{(a,b),(0,0)\})$ and Lie$(\{(a,b)\},\Delta)$ have derived length~$2$.
\end{lem}
\begin{proof}
Let $(a,b)\in \mathcal{M}$.
Then by the commutation relation~\eqref{CR1}, it follows that either Lie($\{(a,b), (0,0)\}$) is Abelian (when $a=b$) or it is isomorphic to the non-abelian Lie algebra of dimension~$2$ (when $a\neq b$).
Therefore, Lie($\{(a,b), (0,0)\}$) is either Abelian or solvable of derived length~$2$.

Next, by commutation relation~\eqref{CR3}, it follows that Lie($\{(a,b),\Delta\}$) is either Abelian (when $a+b=0$) or solvable of derived length~$2$ (when $a+b\neq 0$).	
\end{proof}
Next, we consider the case when both homogeneous derivations have non-zero bidegrees.
At first, we prove a criterion for non-solvability (\thref{LNS1}), then we prove a criterion for solvability (\thref{P6}) and at last, we join them to get the main result, \thref{Th1}.
Now, we prove two technical lemmas (\ref{P4} and \ref{P5}) as preparation for \thref{LNS1}.
%; which gives a criterion for non-solvability. 

\begin{lem}\thlabel{P4}
	For any $(a,b), (c,d)\in \mathcal{M}$ and $r,s\in \ZZ$ with $(a,b)\neq (0,0)$ and $r\neq s$ the following statements hold:
	\begin{enumerate}[{\rm(a)}]
		
		\item if $(c,d)\notin \underline{\ell}^{(a,b)}$ then
		$\partial_{((r+s)a+2c, (r+s)b+2d)}\in$
		Lie$(\{(ra+c,rb+d),(sa+c,sb+d)\})\, ;$					
		
		\item if $a+c\neq b+d$ then $\partial_{((r+s)(a+c), (r+s)(b+d))}\in$ Lie$(\{(r(a+c),r(b+d)),{(s(a+c),s(b+d))}\})$.
		
	\end{enumerate}
\end{lem}
\begin{proof}
	{\bf (a) : }
	Note that for any $r,s\in \ZZ$, we have
		\begin{equation}\label{P4E}
			(r-s)	
		\begin{vmatrix}
			a & c+1\\
			b & d+1\\ 
		\end{vmatrix} 
		=
		\begin{vmatrix}
			ra+c+1 & sa+c+1\\
			rb+d+1 & sb+d+1\\ 
		\end{vmatrix}.
		\end{equation}
	
	Let us assume that $(c,d)\notin \underline{\ell}^{(a,b)}$.
	Then $(c,d)\neq (-1,-1).$
	Hence, it follows that $(sa+c,sb+d)\neq (-1,-1)$, for all $s\in \ZZ$.
	Next for $r\neq s$, the determinant on the L.H.S of \eqref{P4E} is non-zero.
	Hence, the determinant on the R.H.S of \eqref{P4E} is also non-zero.
	Therefore, we have $$(ra+c,rb+d)\notin \ell^{(sa+c,sb+d)},\, r\neq s.$$
	Next by \thref{P3}(b), we can conclude that 
	\begin{center}
		$\partial_{((r+s)a+2c, (r+s)b+2d)}\in$
		Lie$(\{(ra+c,rb+d),(sa+c,sb+d)\})$, for $r\neq s$.
	\end{center}
	\smallskip
	\noindent
	{\bf (b) : }
	Note that for any $r,s\in \ZZ$, we have
	$
	%	\begin{array}{lllll}
		(r-s)
		\begin{vmatrix}
			1 & a+c\\
			1 & b+d\\ 
		\end{vmatrix}
		=
		\begin{vmatrix}
			ra+rc+1 & sa+sc+1\\
			rb+rd+1 & sb+sd+1\\ 
		\end{vmatrix}.
	$
	
	Now following the proof of statement (a), we can conclude that
	if $a+c\neq b+d$ and $r\neq s$ then $(s(a+c),s(b+d))\neq (-1,-1)$, for all $s\in \ZZ$ and  
	\begin{center}
		$(r(a+c),r(b+d))\notin \ell^{(s(a+c),s(b+d))},$ for $r\neq s$.
	\end{center}
	Now by \thref{P3}(b), the result follows.
\end{proof}
			
\begin{lem}\thlabel{P5}
	Let $(a,b),(c,d)\in \mathcal{M}$ be such that $c\neq d$, $a+c\neq -1$ and 
	$a+c=b+d$.
	Then for any $r,s\in \ZY$, the following statements hold:
	\begin{enumerate}[\rm(a)]
		\item 
		$\partial_{((r+s)(a+c)+c,(r+s)(b+d)+d)}\in$ Lie$(\{(r(a+c)+c,r(b+d)+d), (s(a+c),s(b+d))\})$\\ and
		$\partial_{((r+s)(a+c)+a,(r+s)(b+d)+b)}
		\in$ Lie$(\{(r(a+c)+a, r(b+d)+b),(s(a+c),s(b+d))\})\, ;$
		
		\item if
		$(sa+(s+1)c, sb+(s+1)d)\in \ell^ {((s+1)a+sc, (s+1)b+sd)}$  then for any $r\neq s$, we have\\
		$\partial_{((2r+1)(a+c), (2r+1)(b+d))}\in$ Lie$(\{((r+1)a+rc, (r+1)b+rd),ra+(r+1)c, rb+(r+1)d\})$.
	\end{enumerate}
\end{lem}

\begin{proof}
	Since $a+c=b+d$, we have
	\begin{equation}\label{E8}
		0=
		\begin{vmatrix}
			1 & a+c\\
			1 & b+d\\ 
		\end{vmatrix}= 
		\begin{vmatrix}
			1 & c\\
			1 & d
		\end{vmatrix}-
		\begin{vmatrix}
			a & 1\\
			b & 1
		\end{vmatrix}.
	\end{equation}
	{\bf (a) : }
	Note that for any $r,s\in \ZY$,
	\begin{equation}\label{E9}
		\begin{vmatrix}
			s(a+c)+1 & r(a+c)+c+1\\
			s(b+d)+1 & r(b+d)+d+1
		\end{vmatrix}=
		s\begin{vmatrix}
			a&c\\
			b&d
		\end{vmatrix}
		+(s-r)\begin{vmatrix}
			a & 1\\
			b & 1
		\end{vmatrix}
		+(r+1-s)\begin{vmatrix}
			1 & c\\
			1 & d
		\end{vmatrix}
	\end{equation} and 
	\begin{equation*}
		\begin{vmatrix}
			s(a+c)+1 & r(a+c)+a+1\\
			s(b+d)+1 & r(b+d)+b+1
		\end{vmatrix}=
		-s\begin{vmatrix}
			a&c\\
			b&d
		\end{vmatrix}
		+(s-r-1)\begin{vmatrix}
			a & 1\\
			b & 1
		\end{vmatrix}
		+(r-s)\begin{vmatrix}
			1 & c\\
			1 & d
		\end{vmatrix}.
	\end{equation*}
	Therefore, by (\ref{E8}), we have
	\begin{equation*}
		\begin{vmatrix}
			s(a+c)+1 & r(a+c)+c+1\\
			s(b+d)+1 & r(b+d)+d+1
		\end{vmatrix}=-	\begin{vmatrix}
			s(a+c)+1 & r(a+c)+a+1\\
			s(b+d)+1 & r(b+d)+b+1
		\end{vmatrix}
	\end{equation*}
	Now $a+c=b+d\neq -1$ and $s\in \ZY$; hence, $(sa+sc,sb+sd)\neq (-1,-1)$.
	Thus,
	$$(r(a+c)+c,r(a+c)+d)\in \ell^{(sa+sc,sb+sd)} \iff (r(a+c)+a, r(a+c)+b)\in~ \ell^{(sa+sc,sb+sd)}.$$
	
	Suppose, if possible, that $(r(a+c)+c,r(a+c)+d)\in \ell^{(sa+sc,sb+sd)}$.
	Then the determinant on L.H.S of (\ref{E9}) is zero and hence by (\ref{E8}), we have 
	\begin{equation}\label{E10}
		s\begin{vmatrix}
			a&c\\
			b&d
		\end{vmatrix}
		+\begin{vmatrix}
			1 & c\\
			1 & d
		\end{vmatrix}=0.
	\end{equation}
	\noindent
	Now we also have $ (r(a+c)+a, r(a+c)+b)\in~ \ell^{(sa+sc,sb+sd)}$.
%	By assumption $c\neq d$;
	Hence, we can conclude that
	$(r(a+c)+a, r(a+c)+b)\in~ \ell^{(r(a+c)+c,r(a+c)+d)}$.
	Therefore,
	\begin{equation}\label{E2}
		\begin{vmatrix}
			r(a+c)+c+1 & r(a+c)+a+1 \\
			r(a+c)+d+1 & r(a+c)+b+1 
		\end{vmatrix}=0.
	\end{equation}
	Next, by equations \eqref{E8} and \eqref{E2}, we have
	\begin{equation}\label{E11}
		(2r+1)\begin{vmatrix}
			a & c\\
			b & d
		\end{vmatrix}+2
		\begin{vmatrix}
			1 & c\\
			1 & d
		\end{vmatrix}=0.
	\end{equation} 
	Therefore, from (\ref{E10}) and (\ref{E11}), we have $\begin{vmatrix}
		1 & c\\
		1 & d
	\end{vmatrix}= 0$, i.e., $c=d$, a contradiction. 
	Hence, $$(r(a+c)+c,r(a+c)+d), (r(a+c)+a, r(a+c)+b)\notin \ell^{(sa+sc,sb+sd)}.$$
	Now, by \thref{P3}(b)
	the result follows.
	
	\smallskip
	\noindent
	{\bf (b) : }
	Note that for any $s\in \ZY$,
	\begin{center}
		$(sa+(s+1)c, sb+(s+1)d)\neq (-1,-1)$ and $(s+1)a+sc, (s+1)b+sd\neq (-1,-1)$.
	\end{center}
	By assumption $(sa+(s+1)c, sb+(s+1)d)\in \ell^ {((s+1)a+sc, (s+1)b+sd)}$;
	hence, we have
	\begin{equation}\label{E121}
		\begin{vmatrix}
			sa+(s+1)c+1 & (s+1)a+sc+1 \\
			sb+(s+1)d+1 & (s+1)b+sd+1 
		\end{vmatrix}=0.
	\end{equation}
	Therefore, by \eqref{E8} and \eqref{E121}, we can conclude that 
		\begin{equation}\label{E12}
		(2s+1)
		\begin{vmatrix}
			a & c\\
			b & d
		\end{vmatrix}+2
		\begin{vmatrix}
			1 & c\\
			1 & d
		\end{vmatrix}=0.
	\end{equation}
	Suppose, if possible, that $ (ra+(r+1)c,rb+(r+1)d)\in \ell^{((r+1)a+rc,(r+1)b+rd)}$ for some $r\neq s$.
	Then, by \eqref{E8} we have 
	\begin{equation}\label{E13}
%		0=\begin{vmatrix}
%			ra+(r+1)c+1 & (r+1)a+rc+1 \\
%			rb+(r+1)d+1 & (r+1)b+rd+1 
%		\end{vmatrix}=-
        (2r+1)
		\begin{vmatrix}
			a & c\\
			b & d
		\end{vmatrix}+2
		\begin{vmatrix}
			1 & c\\
			1 & d
		\end{vmatrix}=0.
	\end{equation}
	Therefore, by (\ref{E12}) and (\ref{E13}), we have
	$\begin{vmatrix}
		1 & c\\
		1 & d
	\end{vmatrix} =0$, i.e., $c=d$,	a contradiction.
	Hence,
	$$ (ra+(r+1)c,rb+(r+1)d)\notin \ell^{((r+1)a+rc,(r+1)b+rd)}$$
	and now by \thref{P3}(b), the result follows.
\end{proof}

Our next result
%we consider $(a,b), (c,d)\in \mathcal{M}\setminus\{(0,0)\}$ and
provides a criterion for non-solvability of Lie$(\{(a,b),(c,d)\})$.
It is a key lemma of this paper.
\begin{lem}\thlabel{LNS1}
	For any $(a,b),(c,d)\in \mathcal{M}\setminus \{(0,0)\}$, if $(a,b)\not\in \ell^{(c,d)}$, $(a,b)\notin \underline{\ell}^{(c,d)}$ and  $(c,d)\notin \underline{\ell}^{(a,b)}$ then
	Lie$(\{(a,b),(c,d)\})$ is non-solvable.	
\end{lem}
\begin{proof}
	Let $\mathfrak{L}$ denote the subalgebra Lie$(\{(a,b),(c,d)\})$.
	We will show that $\mathfrak{L}^{(4)}\neq 0$ and hence by \thref{ThDL} it will follow that $\mathfrak{L}$ is non-solvable.
	
	By assumption $(a,b)\notin \ell^{(c,d)}$; hence,  by \thref{P3}(b), we have 
	\begin{equation}\label{F1}
		\partial_{(a+c,b+d)}\in~ \mathfrak{L}^{(1)}.
	\end{equation}
	Now $(c,d)\notin \underline{\ell}^{(a,b)}$ and $\partial_{(c,d)},\partial_{(a+c,b+d)}\in \mathfrak{L}$; hence, by \thref{P4}(a),
	we have 
	\begin{equation}\label{F2}
		\partial_{(a+2c,b+2d)}\in \mathfrak{L}^{(1)}.
	\end{equation}
	Similarly, with the help of $(a,b)\notin \underline{\ell}^{(c,d)}$, we can conclude that 
	\begin{equation}\label{F3}
		\partial_{(2a+c,2b+d)}\in \mathfrak{L}^{(1)}.
	\end{equation}
	Next $(a,b)\notin \underline{\ell}^{(c,d)}$ and $\partial_{(a,b)},\partial_{(a+2c,b+2d)}\in \mathfrak{L}$; hence, by \thref{P4}(a), we have 
	\begin{equation}
		\partial_{(2a+2c,2b+2d)}\in\mathfrak{L}^{(1)}.
	\end{equation}
	Now we divide the proof in two cases and show that in each one of them 
	$\mathfrak{L}^{(4)}\neq 0$.\\
	
	\noindent
	{\bf Case I : } Suppose $a+c\neq b+d$.\\
	\noindent
	Then applying \thref{P4}(b) for 
	$\partial_{(a+c,b+d)}, \partial_{(2a+2c,2b+2d)}\in\mathfrak{L}^{(1)}$, we get
	$$\partial_{(3a+3c,3b+3d)}\in ~\mathfrak{L}^{(2)}.$$
	Next by applying  \thref{P4}(b), first for $\partial_{(a+c,b+d)}, \partial_{(3a+3c,3b+3d)}, \partial_{(2a+2c,2b+2d)} \in\mathfrak{L}^{(1)}$ \\consecutively, we get
	$$\partial_{(4a+4c,4b+4d)}, \partial_{(5a+5c,5b+5d)}\in~\mathfrak{L}^{(2)},$$
	and then, for $\partial_{(4a+4c,4b+4d)},\partial_{(3a+3c,3b+3d)}, \partial_{(5a+5c,5b+5d)}\in~\mathfrak{L}^{(2)}$ consecutively, we get $$\partial_{(7a+7c,7b+7d)},\partial_{(8a+8c,8b+8d)}\in \mathfrak{L}^{(3)},$$
	and at last, for $\partial_{(7a+7c,7b+7d)},\partial_{(8a+8c,8b+8d)}\in \mathfrak{L}^{(3)},$ we have
	$$\partial_{(15a+15c,15b+15d)}\in \mathfrak{L}^{(4)}.$$
	
	\noindent
	{\bf Case II : } Suppose $a+c=b+d$.
	
	\noindent
	Note that $(a,b)\notin \underline{\ell}^{(c,d)}$ hence $\{(a,b),(c,d)\}\neq \{(0,-1), (-1,0)\}$.
	Therefore, $a+c\geqslant 0$.
	Now $(a,b)\notin \underline{\ell}^{(c,d)}$; therefore,  either $a\neq b$ or $c\neq d$.
	By the condition $a+c=b+d$, we have both
		$a\neq b$ and $c\neq d$.	
	Next  $\partial_{(a+c,b+d)},\partial_{(a+2c,b+2d)}\in \mathfrak{L}^{(1)}$ (cf. \eqref{F1} and \eqref{F2}); hence, by \thref{P4}(a), we have $$\partial_{(2a+3c,2b+3d)}\in \mathfrak{L}^{(2)}.$$
	Similarly, with the help of \eqref{F3} instead of \eqref{F2}, we can conclude that $\partial_{(3a+2c,3b+2d)}\in \mathfrak{L}^{(2)}.$
	
	We again split the proof in two cases and show that in each one of them 
	$\mathfrak{L}^{(4)}\neq 0$.
	
	\smallskip
	\noindent
	{\bf Case II(a) : }Suppose $(2a+3c,2b+3d)\notin \ell^ {(3a+2c,3b+2d)}$.
	
	\noindent
	Then by \thref{P3}(b) and $\partial_{(2a+3c,2b+3d)},\partial_{(3a+2c,3b+2d)}\in \mathfrak{L}^{(2)}$, we have $$\partial_{(5a+5c,5b+5d)}\in\mathfrak{L}^{(3)}.$$
	Now by applying \thref{P5}(a), first for $\partial_{(2a+3c,2b+3d)},\partial_{(5a+5c,5b+5d)}\in \mathfrak{L}^{(2)}$, we have 
	$$\partial_{(7a+8b,7c+8d)}\in \mathfrak{L}^{(3)},$$
	and then for $\partial_{(7a+8b,7c+8d)},\partial_{(5a+5c,5b+5d)}\in \mathfrak{L}^{(3)},$ we have  $$\partial_{(12a+13c,12b+13d)}\in \mathfrak{L}^{(4)}.$$

	\smallskip
	\noindent
	{\bf Case II(b) : }Suppose $(2a+3c,2b+3d)\in \ell^ {(3a+2c,3b+2d)}$.\\
	Then by \thref{P5}(b) and $\partial_{(a+2c,b+2d)},\partial_{(2a+c,2b+d)}\in \mathfrak{L}^{(1)}$ (cf. \eqref{F2} and \eqref{F3}),
	we have
	$$\partial_{(3a+3c,3b+3d)}\in \mathfrak{L}^{(2)}.$$ 
	Next by
	applying \thref{P5}(a), first for $\partial_{(a+2c,b+2d)},\partial_{(2a+c,2b+d)},\partial_{(3a+3c,3b+3d)}\in \mathfrak{L}^{(1)}$, we have 
	$$\partial_{(4a+5c,4b+5d)}, \partial_{(5a+4c,5b+4d)}\in \mathfrak{L}^{(2)},$$
	and then, for $\partial_{(4a+5c,4b+5d)}, \partial_{(5a+4c,5b+4d)}, \partial_{(3a+3c,3b+3d)}\in \mathfrak{L}^{(2)}$, we have
	$$\partial_{(7a+8c,7b+8d)},\partial_{(8a+7c,8b+7d)}\in \mathfrak{L}^{(3)}.$$
	Note that in this case $(2a+3c,2b+3d)\in \ell^ {(3a+2c,3b+2d)}$; hence, from \thref{P5}(b) and $\partial_{(7a+8c,7b+8d)},\partial_{(8a+7c,8b+7d)}\in \mathfrak{L}^{(3)}$,  we have $$\partial_{(15a+15c,15b+15d)}\in \mathfrak{L}^{(4)}.$$
\end{proof}

Now we prove a result concerning solvability of Lie($\{(a,b),(c,d)\}$).

\begin{lem}\thlabel{P6}
	Suppose $(a,b),(c,d)\in \mathcal{M}$ is such that
	$\{(a,b), (c,d)\}\neq \{(0,-1),(-1,0)\}$, $(a,b)\neq (0,0)$, $(a,b)\notin \ell^{(c,d)}$ and $(c,d)\in \underline{\ell}^{(a,b)}$.
	Then Lie$(\{(a,b),(c,d)\})$ is solvable of derived length $2$.
	
	Moreover,
	let $\mathcal{S}\subseteq\mathcal{M}$ be such that 
%\begin{center}
%	
%	Lie$(\{(a,b),(c,d)\})$ = Lie$(\mathcal{S})$ and
	Lie$(\{(a,b),(c,d)\})^{(1)}$ = Lie$(\mathcal{S})$.
%	\end{center} 
	 Then
	 \begin{center}
	 	$\mathcal{S}= \{((r+1)a+c,(r+1)b+d)\in \mathcal{M}\mid r\in \ZZ\}\subseteq \ell^{(c,d)}=\underline{\ell}^{(a,b)}$
	 \end{center}
	 and
	\begin{center}
		Lie$(\{(a,b),(c,d)\})$ = $\overunderset{}{(e,f)\in \mathcal{S}}{\oplus}\kk\partial_{(e,f)}\oplus \kk\partial_{(c,d)}\oplus \kk\partial_{(a,b)}$.
	\end{center}
%	\begin{center}
%		$\mathcal{S}\setminus\{(a,b)\}\subseteq\{(ra+c,rb+d)\in \mathcal{M}\mid r\in \ZZ\}\subseteq \ell^{(c,d)}=\underline{\ell}^{(a,b)}$
%	\end{center}
%	and 
	
%	
%	
%		 where the summation runs over all $r\in \ZZ$, such that $(ra+c,rb+d)\in \mathcal{M}$. 
%	Moreover, 
%Lie$(\{(a,b),(c,d)\})$ is solvable of derived length $2$.	
\end{lem}
\begin{proof}
	Let $\mathfrak{L}$ denote the subalgebra Lie$(\{(a,b),(c,d)\})$.
%	Suppose $(a,b)\in \mathcal{M}$, $(c,d)\in \mathcal{M}\setminus\{(0,0)\}$ and $\mathfrak{L}$:= Lie($\{(a,b), (c,d)\}$).
	Now $(c,d)\in \underline{\ell}^{(a,b)}$; hence, by \thref{P1}, we have 
	\begin{equation}\label{EP6}
	(ra+c,rb+d)\in \ell^{(c,d)}=\underline{\ell}^{(a,b)}, \text{ for all } r\in \ZY.
\end{equation}
	Next, $(a,b)\notin \ell^{(c,d)}$; therefore, $\partial_{(a+c,b+d)}\in \mathfrak{L}$ (cf. \thref{P3}(b)) and
	 \begin{equation}\label{EP61}
	 	\ell^{(c,d)}\cap \ell^{(a,b)}\cap\mathcal{M}=\emptyset.
	 \end{equation}
	Hence from \eqref{EP6} and \eqref{EP61}, we have
	\begin{center}
		$(ra+c,rb+d)\notin \ell^{(a,b)}$, for all $r\in \ZY$.
	\end{center}
	Therefore by induction and \thref{P3}(b),
	we have $$\partial_{((r+1)a+c,(r+1)b+d)}\in \mathfrak{L}, \forall r\in \ZZ, \text{ whenever } ((r+1)a+c,(r+1)b+d)\in \mathcal{M}.$$
	Let $\mathcal{S}= \{((r+1)a+c,(r+1)b+d)\in \mathcal{M}\mid r\in \ZZ\}.$
	Then by \eqref{EP6}, $\mathcal{S}\subseteq \ell^{(c,d)}=\underline{\ell}^{(a,b)}$ and it also follows that $(ra+c,rb+d)\in \ell^{(sa+c,sb+d)}$, for all $r,s\in \ZY$.
	Hence,  for any $(sa+c,sb+d), (ra+c,rb+d)\in \mathcal{M}$, we have
	\begin{equation}\label{EE2}
		[\partial_{(c,d)}, \partial_{(ra+c,rb+d)}]=0=[\partial_{(sa+c,sb+d)}, \partial_{(ra+c,rb+d)}]\, (\text{by  \eqref{EP6} and \thref{P3}(a)}). 	
	\end{equation}
	Thus,
\begin{center}
		$\mathfrak{L}=\overunderset{}{(e,f)\in \mathcal{S}}{\oplus}\kk\partial_{(e,f)}\oplus \kk\partial_{(c,d)}\oplus \kk\partial_{(a,b)}$
%\end{center}
	and 
%	\begin{equation*}
		$\mathfrak{L}^{(1)}=\overunderset{}{(e,f)\in \mathcal{S}}{\oplus}\kk\partial_{(e,f)}.$
	\end{center}
	Next by \eqref{EE2}, we have $\mathfrak{L}^{(2)}=0$.	
\end{proof}
%	\begin{rem}\thlabel{P7}
%	{\rm Consider all the assumptions of the above result and
%		let $\mathcal{S}, \mathcal{S}_1\subseteq\mathcal{M}$ be such that Lie($\{(a,b),(c,d)\}$) = Lie($\mathcal{S}$) and Lie($\{(a,b),(c,d)\}$)$^{(1)}$ = Lie($\mathcal{S}_1$). 
%		Then, the above proof of \thref{P6} reveals that 
%		\begin{center}
%			$\mathcal{S}\setminus\{(a,b)\}=\{(ra+c,rb+d)\in \mathcal{M}\mid r\in \ZZ\}\subseteq \ell^{(c,d)}=\underline{\ell}^{(a,b)}$
%		\end{center}
%		and 
%		\begin{center}
%			$\mathcal{S}_1= \{(ra+c,rb+d)\in \mathcal{M}\mid r\in \ZY\}\subseteq \ell^{(c,d)}=\underline{\ell}^{(a,b)}$.
%		\end{center}
%	}
%\end{rem}
%\noindent
Note that for any $m\in \ZY$, $(m,-1)\notin\ell^{(-1,0)}$ and $(m,-1)\in\underline{\ell}^{(-1,0)}=\ell^{(0,-1)}$.
Therefore, from the above result (\thref{P6}), we can conclude that
%\begin{center}
%	Lie$(\{(a,b),(c,d)\})$ = $\overunderset{\infty}{r=0}{\oplus}\kk\partial_{(a+rc,b+rd)}\oplus \kk\partial_{(c,d)}$, for $(a,b),(c,d)\in \ZZ^2$,
%\end{center}
\begin{center}
	Lie$(\{(m,-1),(-1,0)\})$ = $\overunderset{m}{r=0}{\oplus}\kk\partial_{(r,-1)}\oplus \kk\partial_{(-1,0)}$, for $m\in \ZY$,
\end{center}
%\begin{center}
%	Lie$(\{(m,-1),(c,0)\})$ = $\overunderset{\infty}{r=0}{\oplus}\kk\partial_{(m+rc,-1)}\oplus \kk\partial_{(c,0)}$, for $c\in \ZY,\, m\in \ZZ$,
%\end{center}
%\begin{center}
%	Lie$(\{(0,-1),(-1,m)\})$ = $\overunderset{m}{r=0}{\oplus}\kk\partial_{(-1,r)}\oplus \kk\partial_{(0,-1)}$, for $m\in \ZY$,
%\end{center}
%and
%\begin{center}
%	Lie$(\{(-1,m),(0,d)\})$ = $\overunderset{\infty}{r=0}{\oplus}\kk\partial_{(-1,m+rd)}\oplus \kk\partial_{(0,d)}$, for $d\in \ZY,\, m\in \ZZ$.
%\end{center}

Next, we provide a necessary and sufficient criterion for solvability  of a  subalgebra generated by two homogeneous derivations with non-zero bidegrees.
% not equal to ${(0,0)}$.

\begin{prop}\thlabel{Th1}
	Suppose $(a,b), (c,d)\in \mathcal{M}\setminus \{(0,0)\}$.
%	 is such that $\{(a,b),(c,d)\}\neq\{(0,-1),(-1,0)\}$.
	Then the following statements hold:
%	Lie$(\{(a,b),(c,d)\})$ is solvable if and only if one of the two statements hold:
	\begin{enumerate}[\rm(a)]
		\item Lie$(\{(a,b),(c,d)\})$ is Abelian if and only if either $\{(a,b),(c,d)\}=\{(0,-1),(-1,0)\}$ or $(a,b)\in \ell^{(c,d)}$;
		\item  Lie$(\{(a,b),(c,d)\})$ is non-abelian and solvable if and only if $(a,b)\notin \ell^{(c,d)}$,\\ $\{(a,b),(c,d)\}\neq\{(0,-1),(-1,0)\}$ and either 
		$(c,d)\in\underline{\ell}^{(a,b)}$ or $(a,b)\in~\underline{\ell}^{(c,d)}$.
		Here, Lie$(\{(a,b),(c,d)\})$ has derived length~$2$.
	\end{enumerate}
%	 $(a,b)\notin \ell^{(c,d)}$ and either
%	$(a,b)\in~\underline{\ell}^{(c,d)}$ 
%	or $(c,d)\in\underline{\ell}^{(a,b)}$.
%	Here, Lie$(\{(a,b),(c,d)\})$ is solvable of derived length~$2$.	
\end{prop}
\begin{proof}
	{\bf (a) : }Follows from \thref{RAS1}.
	
	{\bf (b) : }
	Let $\mathfrak{L}$ denote the subalgebra Lie($\{(a,b), (c,d)\}$).
	Suppose $\mathfrak{L}$ is non-abelian and solvable.
	Then by \thref{RAS1}, we have $\{(a,b),(c,d)\}=\{(0,-1),(-1,0)\}$ and $(a,b)\notin \ell^{(c,d)}$.
    If $(a,b)\notin \underline{\ell}^{(c,d)}$ 
	and $(c,d)\notin\underline{\ell}^{(a,b)}$ hold simultaneously, then
    by \thref{LNS1}, we can conclude that $\mathfrak{L}$ is non-solvable.
    Therefore, either $(a,b)\in~\underline{\ell}^{(c,d)}$ 
    	or $(c,d)\in\underline{\ell}^{(a,b)}$.

    Next we prove the converse.
	From the condition $\{(a,b),(c,d)\}\neq\{(0,-1),(-1,0)\}$ and $(a,b)\notin 
	\ell^{(c,d)}$, it follows that
	$\mathfrak{L}$ is non-abelian (cf. \thref{RAS1}).
	Next without loss of generality, we assume that 
	$(c,d)\in \underline{\ell}^{(a,b)}$.
	Now the result follows from \thref{P6}.
	\end{proof}
%	
%	
%	\begin{rem}\thlabel{P7}
%		{\rm Consider $(a,b),(c,d)\in \mathcal{M}\setminus\{(0,0)\}$ with $\{(a,b),(c,d)\}\neq \{(0,-1),(-1,0)\}$, $(a,b)\notin \ell^{(c,d)}$ and $(c,d)\in \underline{\ell}^{(a,b)}$.
%			Moreover,
%			let $\mathcal{S}, \mathcal{S}_1\subseteq\mathcal{M}$ be such that \begin{center}
%				Lie($\{(a,b),(c,d)\}$) = Lie($\mathcal{S}$) and Lie($\{(a,b),(c,d)\}$)$^{(1)}$ = Lie($\mathcal{S}_1$).
%			\end{center} 
%			Then the above proof of \thref{P6} (cf. \eqref{EP6}) reveals that 
%			\begin{center}
%				$\mathcal{S}\setminus\{(a,b)\}\subseteq\{(ra+c,rb+d)\in \mathcal{M}\mid r\in \ZZ\}\subseteq \ell^{(c,d)}=\underline{\ell}^{(a,b)}$
%			\end{center}
%			and 
%			\begin{center}
%				$\mathcal{S}_1= \{(ra+c,rb+d)\in \mathcal{M}\mid r\in \ZY\}\subseteq \ell^{(c,d)}=\underline{\ell}^{(a,b)}$.
%			\end{center}
%		}
%	\end{rem}
%	\begin{rem}\thlabel{RSSTR}
%		{\rm 
%		Combining \thref{RSST0} and \thref{Th1} we get the following criterion of solvability: 
%		Suppose $(a,b), (c,d)\in \mathcal{M}$ and $\{(a,b),(c,d)\}\neq\{(0,-1),(-1,0)\}$.
%		Then Lie$(\{(a,b),(c,d)\})$ is Abelian if and only if $(a,b)\in \ell^{(c,d)}$.
%		 Next, Lie$(\{(a,b),(c,d)\})$ is non-abelian and solvable if and only if $(a,b)\notin \ell^{(c,d)}$ and either
%		\begin{center}
%			$\begin{vmatrix}
%				a & c+1\\
%				b & d+1
%			\end{vmatrix}=0\,$     or     $\,\begin{vmatrix}
%			a+1 & c\\
%			b+1 & d
%			\end{vmatrix}=0.$
%		\end{center} 
%		Here, Lie$(\{(a,b),(c,d)\})$ is solvable of derived length~$2$.
%		}
%	\end{rem}

	\subsection{General case}\label{GC}
	In this subsection, our goal is to give a criteria of solvability for a regular subalgebra (\thref{GC4}).
%	Before that, we prove a technical lemma in this regard.
At first, we proof a criterion of non-solvability for a regular subalgebra generated by three homogeneous derivations with non-zero bidegrees.

	\begin{lem}\thlabel{GC1}
		Suppose for three distinct points $(a,b), (c,d), (e,f)\in \mathcal{M}\setminus\{(0,0)\}$ the following conditions hold:
		\begin{enumerate}[\rm(a)]
			\item the sets $\{(a,b),(c,d)\}$ and $\{(c,d),(e,f)\}$ are not equal to $\{(0,-1),(-1,0)\}\, ;$
			\item $(a,b), (e,f)\notin \ell^{(c,d)}\, ;$ 
			\item $(c,d)\in \underline{\ell}^{(a,b)}$ and $(e,f)\in \underline{\ell}^{(c,d)}$.
		\end{enumerate}
		Then, the subalgebra Lie$(\{(a,b), (c,d), (e,f)\})$ is non-solvable.	
	\end{lem}
	\begin{proof}
		Note that $\{(a,b),(c,d)\}\neq\{(0,-1),(-1,0)\}$ (condition~(a)), $(a,b)\notin \ell^{(c,d)}$ (condition~(b)) and $(c,d)\in \underline{\ell}^{(a,b)}$ (condition~(c)); hence, by \thref{P2}, we have
		\begin{equation}\label{GCE1}
			(a,b)\notin \underline{\ell}^{(c,d)}.
		\end{equation}
%		Similarly, from the assumptions we can conclude that 
%		\begin{equation}\label{GCE2}
%			(c,d)\notin \underline{\ell}^{(e,f)}.
%		\end{equation}
		
		Let $\mathfrak{L}$ denote the subalgebra Lie$(\{(a,b),(c,d),(e,f)\})$ and $\mathfrak{L}_1$ denote the subalgebra Lie$(\{(a,b),(e,f)\})$.
		If $\mathfrak{L}_1$ is non-solvable then $\mathfrak{L}$ is non-solvable and we are done.
		
		\smallskip
		So henceforth, we assume that $\mathfrak{L}_1$ is solvable.
		Now, we will try to understand the relation between $(a,b)$ and $(e,f)$.
		
%		If $\{(a,b),(e,f)\}= \{(0,-1),(-1,0)\}$ then $\underline{\ell}^{(a,b)}=\ell^{(e,f)}$.
%		Hence, by condition~(c), we have $(c,d)\in \underline{\ell}^{(a,b)}=\ell^{(e,f)}$.
%		But it contradicts the fact that $(c,d)\in \underline{\ell}^{(a,b)}$ (condition~(c)) and $(c,d)\in \ell^{(e,f)}$ (condition~(b)).
%		Therefore, 
%		\begin{equation}\label{GCE3}
%			\{(a,b),(e,f)\}\neq\{(0,-1),(-1,0)\}.
%		\end{equation}
		
		By \eqref{GCE1}, we have $(a,b)\notin \underline{\ell}^{(c,d)}$ and by condition~(c), we know that $\ell^{(e,f)}= \underline{\ell}^{(c,d)}$; hence,
		\begin{equation}\label{GCE4}
			(a,b)\notin\ell^{(e,f)}.
		\end{equation}
		If $(e,f)\in \underline{\ell}^{(a,b)}$, then by condition~(c), we have $(c,d),(e,f)\in \underline{\ell}^{(a,b)}$, i.e., $(e,f)\in \ell^{(c,d)}$, a contradiction to condition~(b). 
		 Therefore,
		 \begin{equation}\label{GCE5}
		 	(e,f)\notin \underline{\ell}^{(a,b)}.
		 \end{equation}
		 Hence, it follows that
		 \begin{equation}\label{GCE3}
		 \{(a,b),(e,f)\}\neq\{(0,-1),(-1,0)\}.
		 \end{equation}
		Now, $\mathfrak{L}_1$ is solvable; hence by \eqref{GCE4}, \eqref{GCE5}, \eqref{GCE3} and \thref{Th1}, we can conclude that \begin{equation}\label{GCE6}
			(a,b)\in \underline{\ell}^{(e,f)}.
		\end{equation}
		Therefore, by \thref{P1}, we have 
		\begin{equation}\label{GCE31}
			(a+e,b+f)\in \ell^{(a,b)}.
		\end{equation}
		Note that by \eqref{GCE4} and \thref{P3}(b), we can conclude that $\partial_{(a+e,b+f)}\in \mathfrak{L}_1\subseteq \mathfrak{L}$.
		Let $\mathfrak{L}_2$ denote the subalgebra Lie$(\{(a+e,b+f),(c,d)\})$.
%		
%		, $(a+e,b+f)\in \underline{\ell}^{(e,f)}$ (by \thref{P7}).
		Next, with the help of \thref{LNS1}, we are going to show that $\mathfrak{L}_2$ is non-solvable and hence it will follow that $\mathfrak{L}$ is non-solvable.
		
		Note that $(a+e,b+f),(c,d)\in \mathcal{M}\setminus\{(0,0)\}$.
		Now if $(a+e,b+f)\in \ell^{(c,d)}$ then by \eqref{GCE31}, we have $\ell^{(a,b)}=\ell^{(a+e,b+f)}= \ell^{(c,d)}$, i.e., $(a,b)\in \ell^{(c,d)}$.
		This contradicts condition~(b); hence,
		\begin{equation}\label{GCE7}
			(a+e,b+f)\notin \ell^{(c,d)}.
		\end{equation}
		If $(a+e,b+f)\in \underline{\ell}^{(c,d)}$ then by condition~(c), we have $(a+e,b+f)\in \underline{\ell}^{(c,d)}=\ell^{(e,f)}.$
		Now, by \thref{P1}, we can conclude that $(e,f)\in \underline{\ell}^{(a,b)}$, a contradiction to \eqref{GCE5}.
		Hence,
		\begin{equation}\label{GCE8}
			(a+e,b+f)\notin \underline{\ell}^{(c,d)}.
		\end{equation}
		If $(c,d)\in \underline{\ell}^{(a+e,b+f)}$, then by condition~(c), we have $\underline{\ell}^{(a,b)}=\ell^{(c,d)}=\underline{\ell}^{(a+e,b+f)}$.
		Therefore, by \thref{N3}, we have $(a,b),(a+e,b+f)\in \widetilde{\ell}_{(c,d)}$, i.e., $(a,b)\in \ell_{(e,f)}$.
		 Now from \eqref{GCE6}, it follows that $e=f$, i.e., $\ell^{(e,f)}=\underline{\ell}^{(e,f)}$.
		This contradicts the facts that $(a,b)\notin \ell^{(e,f)}$ (cf. \eqref{GCE4}) and $(a,b)\in \underline{\ell}^{(e,f)}$ (cf. \eqref{GCE6}).
		Therefore,
		\begin{equation}\label{GCE9}
			(c,d)\notin \underline{\ell}^{(a+e,b+f)}.
		\end{equation}
%		
%		If $\{(a+e,b+f),(c,d)\}= \{(0,-1),(-1,0)\}$ then $\underline{\ell}^{(c,d)}=\ell^{(a+e,b+f)}$.
%		Therefore, by condition~(c), we have $(e,f)\in \underline{\ell}^{(c,d)}=\ell^{(a+e,b+f)}$, i.e., .
%		Hence, by \thref{P1}, we have 
%		Hence, by condition~(c), we have $(c,d)\in \underline{\ell}^{(a,b)}=\ell^{(e,f)}$.
%		But it contradicts the fact that $(c,d)\in \underline{\ell}^{(a,b)}$ (condition~(c)) and $(c,d)\in \ell^{(e,f)}$ (condition~(b)).
%		Therefore, 
%		\begin{equation}\label{GCE3}
%			\{(a,b),(e,f)\}\neq\{(0,-1),(-1,0)\}.
%		\end{equation}
		Thus, by \eqref{GCE7}, \eqref{GCE8}, \eqref{GCE9} and \thref{LNS1}, we can conclude that $\mathfrak{L}_2$ is non-solvable.
%		Hence it follows that Lie($\{(a,b), (c,d), (e,f)\}$) is non-solvable.
	\end{proof}
Next, we give a criterion of solvability for a non-abelian regular subalgebra generated by homogeneous derivations, with non-zero bidegrees.
	It can be considered a generalization of \thref{Th1}(b) (cf. \thref{P6}).
	\begin{lem}\thlabel{GC2}
		Let $\mathcal{S}\subseteq \mathcal{M}\setminus\{(0,0)\}$ be such that $S\neq \{(0,-1), (-1,0)\}$.
		Then Lie$(\mathcal{S})$ is non-abelian and solvable if and only if there exists $(a,b)\in \mathcal{S}$ such that $a\neq b$,\\ $\mathcal{S}\setminus\{(a,b)\}\neq\emptyset$ and $\mathcal{S}\setminus\{(a,b)\}\subseteq \underline{\ell}^{(a,b)}$.
		Here, Lie$(\mathcal{S})$ is solvable of derived length $2$.	
	\end{lem}
	\begin{proof}
		Suppose, there exists
		\begin{center}
		$(a,b)\in \mathcal{S}$ such that $a\neq b$, $\mathcal{S}\setminus\{(a,b)\}\neq \emptyset$ and $\mathcal{S}\setminus\{(a,b)\}\subseteq \underline{\ell}^{(a,b)}\cap\mathcal{M}$.	
		\end{center}
		By the condition $a\neq b$, it follows that $(a,b)\notin \underline{\ell}^{(a,b)}$.
		Hence, by \thref{RAS1}, we can conclude that Lie($\mathcal{S}$) is non-abelian.
		Now, for any $(c,d)\in \mathcal{S}\setminus\{(a,b)\}$, we have $(c,d)\in~\underline{\ell}^{(a,b)}$.
		Hence, by \thref{P6}, we know that 
		\begin{center}
			Lie$(\{(a,b),(c,d)\})^{(1)}\subseteq\ell^{(c,d)}=\underline{\ell}^{(a,b)}.$
		\end{center}
		Therefore,
		\begin{center}
			Lie($\mathcal{S})^{(1)}\subseteq$ 
%Lie($\mathcal{S}\setminus\{(a,b)\}$) $\subseteq$ 
Lie($\underline{\ell}^{(a,b)}$).	
		\end{center}
		Next by \thref{RAS1}, we know that Lie($\underline{\ell}^{(a,b)}$) is Abelian and hence, it follows that Lie($\mathcal{S})^{(1)}$ is Abelian.
		Therefore, Lie($\mathcal{S}$) is solvable of derived length~$2$.
		
		Now we prove the converse.
		Let Lie$(\mathcal{S})$ be non-abelian and solvable.
		Hence there \\ exist $(a,b), (c,d)\in \mathcal{S}$ such that Lie($\{(a,b),(c,d)\}$) is non-abelian and solvable.
		Then by \thref{RAS1}, we have 
		\begin{equation}\label{GCE10}
		\{(a,b),(c,d)\}\neq \{(0,-1),(-1,0)\}	
		\end{equation}
		and
		\begin{equation}\label{GCE11}
		(a,b)\notin \ell^{(c,d)}.	
		\end{equation} 
		Next by \thref{Th1}(b), either $(a,b)\in \underline{\ell}^{(c,d)}$ or $(c,d)\in \underline{\ell}^{(a,b)}$.
		Without loss of generality, we assume that 
		\begin{equation}\label{GCE12}
		(c,d)\in \underline{\ell}^{(a,b)}.	
		\end{equation}
		From \eqref{GCE11} and \eqref{GCE12}, it follows that 
		\begin{equation}
			a\neq b.
		\end{equation}
		Now, we have to show that $\mathcal{S}\setminus\{(a,b)\}\subseteq \underline{\ell}^{(a,b)}\cap\mathcal{M}$.
		We will prove that in two cases.
		
		\noindent
		{\bf Case I : } Suppose Lie($\mathcal{S}$) = Lie$(\{(a,b),(c,d)\})$.\\
		\noindent
		Then the result follows from \thref{P6}. 
		
		\noindent
		{\bf Case II : }
	 Lie$(\mathcal{S})\setminus\text{Lie}(\{(a,b),(c,d)\})\neq \emptyset$.\\
	\noindent
	Let $(e,f)\in\mathcal{S}$ be such that $\partial_{(e,f)}\in$ Lie$(\mathcal{S})$ and  $\partial_{(e,f)}\notin$ Lie$(\{(a,b),(c,d)\})$.
	Then, we will show that 
	\begin{equation}\label{GCE13}
		(e,f)\in \ell^{(c,d)}=\underline{\ell}^{(a,b)}.
	%	\text{ (cf. \eqref{GCE11}) }.
	\end{equation}

	Let $\mathfrak{L}$ denote the subalgebra Lie$(\{(a,b),(c,d),(e,f)\})$ and $\mathfrak{L}_1$ denote the subalgebra Lie$(\{(c,d),(e,f)\})$.
    Then both $\mathfrak{L}$ and $\mathfrak{L}_1$ are solvable.
    Consider $\mathfrak{L}_1$; then, by \thref{Th1}, we have four options either $(e,f)\in \ell^{(c,d)}$ or $\{(c,d),(e,f)\}=\{(0,-1),(-1,0)\}$  or $(e,f)\notin \ell^{(c,d)}$ with $(e,f)\in \underline{\ell}^{(c,d)}$ or $(e,f)\notin \ell^{(c,d)}$ with $(c,d)\in \underline{\ell}^{(e,f)}$.
    Next, we will show that the last three cases can not occur and hence \eqref{GCE13} will follow.\\
% 
%    
%		Then by \thref{GC1},  we can conclude that either  
%		Next we consider the first two cases.\\
		
		\noindent
		{\bf Case (a) : }Suppose, if possible, that $\{(c,d),(e,f)\}=\{(0,-1),(-1,0)\}$.
		
		\noindent
		With out loss of generality, we assume that $(c,d)=(0,-1)$ and $(e,f)=(-1,0)$.
		Then from $(c,d)\in \underline{\ell}^{(a,b)}$ (cf. \eqref{GCE12}), it follows that $b=0$.
		Now $(e,f)\neq (a,b)$, hence $a\in \ZY$.
		Next by the commutation relation~\eqref{CR1}, we have
		$$[\partial_{(e,f)}, \partial_{(a,b)}]=[\partial_{(-1,0)}, \partial_{(a,0)}]=(a+1)\partial_{(a-1,0)}.$$
		Therefore, by induction $\partial_{(-1,0)},\partial_{(1,0)}\in \mathfrak{L}$.
		But Lie($\{(-1,0),(1,0)\}$) is isomorphic to $\mathfrak{sl}_2$ and that makes $\mathfrak{L}$ non-solvable, a contradiction.
		Therefore, 
		\begin{equation}\label{GCE14}
		\{(c,d),(e,f)\}\neq\{(0,-1),(-1,0)\}.	
		\end{equation}
		\noindent
		{\bf Case (b) : }Suppose, if possible, that $(e,f)\notin \ell^{(c,d)}$ and $(e,f)\in \underline{\ell}^{(c,d)}$.\\
		\noindent
		Then by \eqref{GCE10}, \eqref{GCE11}, \eqref{GCE12}, \eqref{GCE14} and \thref{GC1}, we can conclude that $\mathfrak{L}$ is non-solvable, a contradiction.

		\noindent
		{\bf Case (c) : }Suppose, if possible, that $(e,f)\notin \ell^{(c,d)}$ and $(c,d)\in \underline{\ell}^{(e,f)}$.\\
		\noindent
		Let $\mathfrak{L}_2$ denote the subalgebra Lie$(\{(a,b),(e,f)\})$.
		Next, with the help of \thref{LNS1}, we will show that $\mathfrak{L}_2$ is non-solvable and hence we will arrive at a contradiction.
		
		Note that $(c,d)\in \underline{\ell}^{(a,b)}$ (\eqref{GCE12}); therefore by \thref{N3}, we have $(a,b), (e,f)\in \widetilde{\ell}_{(c,d)}$.
		Hence, we can conclude that
		$\ell_{(a,b)}= \ell_{(e,f)}.$
%		Therefore, 
%		\begin{equation}
%			\{(a,b),(e,f)\}\neq\{(0,-1),(-1,0)\}.
%		\end{equation}
%		Consider the subalgebra Lie($\{(a,b),(e,f)\}$).

		Now if $(a,b)\in \ell^{(e,f)}$, then it follows that $e=f$ and hence $\ell^{(e,f)}= \underline{\ell}^{(e,f)}$.
		But it contradicts our assumption that $(c,d)\notin \ell^{(e,f)}$ and $(c,d)\in \underline{\ell}^{(e,f)}$.
		Therefore,
		\begin{equation}\label{GCE15}
			(a,b)\notin \ell^{(e,f)}.	
		\end{equation}
		Next if  $(e,f)\in \underline{\ell}^{(a,b)}$, then by \eqref{GCE12}, we have $\ell^{(c,d)}= \underline{\ell}^{(a,b)}=\ell^{(e,f)}$, i.e., $(e,f)\in \ell^{(c,d)}$, a contradiction.
		Therefore, 
		\begin{equation}\label{GCE16}
			(e,f)\notin \underline{\ell}^{(a,b)}.	
		\end{equation}
		Lastly, if $(a,b)\in \underline{\ell}^{(e,f)}$,
		then $\ell^{(c,d)}=\underline{\ell}^{(e,f)}=\ell^{(a,b)}$.
		Hence, $(a,b)\in \ell^{(c,d)}$, a contradiction to \eqref{GCE11}.
%		 to the fact that Lie($\mathcal{S}$) is solvable.
		Therefore,
		\begin{equation}\label{GCE17}
			(a,b)\notin \underline{\ell}^{(e,f)}.	
		\end{equation} 
		Now, by \eqref{GCE15}, \eqref{GCE16}, \eqref{GCE17} and \thref{LNS1}, it follows that $\mathfrak{L}_2$ is non-solvable, a contradiction.
		
		Next, consider the solvable subalgebra $\mathfrak{L}_1$; by Case~(a), Case~(b) and Case~(c) and \thref{Th1}, it follows that 
		\begin{center}
			 	$(e,f)\in \ell^{(c,d)}=\underline{\ell}^{(a,b)},\,\forall
			 	%	\text{ (cf. \eqref{GCE11}) }.
 (e,f)\in\mathcal{S}$ with $\partial_{(e,f)}\in$ Lie$(\mathcal{S})$ and  $\partial_{(e,f)}\notin$ Lie$(\{(a,b),(c,d)\})$.
		\end{center}
		Now, combining the above result with \thref{P6}, we have 
		$\mathcal{S}\setminus\{(a,b)\}\subseteq \underline{\ell}^{(a,b)}\cap\mathcal{M}.$
	\end{proof}
%	In the above proposition the condition $\underline{\ell}^{(a,b)}\cap \mathcal{M}\neq \emptyset$ is very important.
%	It helps us identify all the regular Borel subalgebras.
%	\begin{rem}\thlabel{GC3}
%		Note that if $(a,b)\in \{(-1,m),(m,-1)\mid m\in \ZY\}$	then $\underline{\ell}^{(a,b)}\cap \mathcal{M}= \emptyset$.
%		For all other $(a,b)$ in $\mathcal{M}$, $\underline{\ell}^{(a,b)}\cap \mathcal{M}\neq \emptyset$.
%	\end{rem}
	Next, we provide a necessary and sufficient condition for solvability  of a  regular subalgebra.
%	Next proposition is the first step towards the main result (\thref{MT1}).
	\begin{prop}\thlabel{GC4}
		Let $\mathfrak{L}$ be a regular subalgebra of Lie(Aut($\mathbb{A}^2$)).
		Then the following statements hold:
	\begin{enumerate}[\rm(a)]
		\item $\mathfrak{L}$ is Abelian if and only if either  $\mathfrak{L}\subseteq$ Lie$(\{(0,0)\},\Delta)$ or
		$\mathfrak{L}\subseteq$ Lie$(\{(1,-1)\},\Delta)$ or
		$\mathfrak{L}\subseteq$ Lie$(\{(-1,1)\},\Delta)$ or $\mathfrak{L}=$ Lie$(\{(0,-1),(-1,0)\})$ or
		$\mathfrak{L}=$ Lie$(\mathcal{S})$, for some $\mathcal{S}\subseteq\mathcal{M}$ such that $\mathcal{S}\subseteq \ell^{(c,d)}$, for all $(c,d)\in \mathcal{S}$;
		\item  $\mathfrak{L}$ is non-abelian and solvable if and only if $\mathfrak{L}^{(1)}$ = Lie($\mathcal{S}$), for some non-empty subset $\mathcal{S}$ of $ \mathcal{M}\setminus\{(0,0)\}$ and either $\mathcal{S}=\{(0,-1),(-1,0)\}$ or $\mathcal{S}\subseteq\ell^{(c,d)}$, for all $(c,d)\in \mathcal{S}$ or there exists $(a,b)\in \mathcal{S}$ such that $a\neq b$, $\mathcal{S}\setminus\{(a,b)\}\neq \emptyset$ and $\mathcal{S}\setminus\{(a,b)\}\subseteq \underline{\ell}^{(a,b)}=\ell^{(c,d)}$, for all $(c,d)\in \mathcal{S}\setminus\{(a,b)\}$. 
		
		\hspace{4mm} Moreover, for the first two cases $\mathfrak{L}$ has derived length~$2$ and for the last case $\mathfrak{L}$ has derived length~$3$.
		%		Moreover, in the last case if $a\neq b$ then $\mathfrak{L}$ has derived length~$3$; otherwise, $\mathfrak{L}$ has derived length~$2$.
	\end{enumerate}	
	\end{prop}
	\begin{proof}
		{\bf (a) : }Follows from \thref{RAS0} and \thref{RAS1}.
		
		{\bf (b) : }Let $\mathfrak{L}$ be non-abelian and solvable.
		Then by non-commutativity, we have
			\begin{center}
			$\mathfrak{L}^{(1)}=$ Lie($\mathcal{S}$), for some $\mathcal{S}\subseteq \mathcal{M}$ and $\mathcal{S}\neq\emptyset$.
		\end{center} 
		Next, by commutation relations \eqref{CR1} and~\eqref{CR3} respectively, we have
		\begin{center}
			$[\partial_{(0,0)}, \partial_{(c,d)}]=(c-d)\partial_{(c,d)}$
			and 
			$[\Delta, \partial_{(c,d)}]=(c+d)\partial_{(c,d)}$, for all $(c,d)\in \mathcal{M}$.
		\end{center}
		Therefore,
		\begin{center}
		 $\mathcal{S}\subseteq \mathcal{M}\setminus\{(0,0)\}$.
		\end{center} 
		Next if $\mathfrak{L}^{(1)}$ is Abelian, then by statement~(a), we have either $\mathcal{S}=\{(0,-1),(-1,0)\}$ or $\mathcal{S}\subseteq\ell^{(c,d)}$, for all $(c,d)\in \mathcal{S}$.
		
		Otherwise, $\mathfrak{L}^{(1)}$ is non-abelian and solvable.
		Then by \thref{GC2}, there exists $(a,b)\in \mathcal{S}$ such that $a\neq b$, $\mathcal{S}\setminus\{(a,b)\}\neq \emptyset$ and $\mathcal{S}\setminus\{(a,b)\}\subseteq \underline{\ell}^{(a,b)}=\ell^{(c,d)}$, for all $(c,d)\in \mathcal{S}\setminus\{(a,b)\}$.
		
		\smallskip
		Now we prove the converse.
		If $\mathfrak{L}^{(1)}$ = Lie($\mathcal{S}$), for some non-empty subset $\mathcal{S}$ of $ \mathcal{M}\setminus\{(0,0)\}$, then $\mathfrak{L}$ is non-abelian.
		
		If $\mathcal{S}=\{(0,-1),(-1,0)\}$ or $\mathcal{S}\subseteq\ell^{(c,d)}$, for all $(c,d)\in \mathcal{S}$, then by \thref{RAS1}, it follows that $\mathfrak{L}^{(1)}$ is Abelian.
		Hence, $\mathfrak{L}$ is solvable of derived length $2$.
		
		Next, if there exists $(a,b)\in \mathcal{S}$ such that $a\neq b$, $\mathcal{S}\setminus\{(a,b)\}\neq \emptyset$ and $\mathcal{S}\setminus\{(a,b)\}\subseteq \underline{\ell}^{(a,b)}=\ell^{(c,d)}$, for all $(c,d)\in \mathcal{S}\setminus\{(a,b)\}$.
		Then by \thref{GC2}, we know that $\mathfrak{L}^{(1)}$ is solvable of derived length $2$.
		Hence, $\mathfrak{L}$ is solvable of derived length~$3$.
	\end{proof}
\section{Regular Borel subalgebras}\label{RBS}
In this section, we prove our main result (\thref{MT1}) about the description of the regular Borel subalgebras of Lie(Aut($\mathbb{A}^2$)) and later discuss their isomorphism classes in Subsection~\ref{IC}.

\subsection{Main Theorem}\label{MT}
Now we are ready to describe the structure of regular Borel subalgebras explicitly.
\begin{thm}\thlabel{MT1}
	\begin{enumerate}[(1)]
		\item Every regular Borel subalgebra of Lie(Aut($\mathbb{A}^2$)) has derived\\ length~$2$~or~$3$.	
		\item The only metabelian regular Borel subalgebra is  of the following form:
		$$Lie(\Delta,\ell_{(1,1)})
		=Lie(\mathfrak{t}_2,\{(c,c)\mid c\in \ZY\})=
		\kk x\dfrac{\partial}{\partial x}\oplus \kk y\dfrac{\partial}{\partial y}\oplus xy\kk[xy]\bigg(x\dfrac{\partial}{\partial x}-y\dfrac{\partial}{\partial y}\bigg).$$ 
%	Moreover, Lie$(\ell_{(1,1)},\Delta)$= $\mathscr{G}(0,0,(i+1)_i,\pmb{0},\pmb{0})$.
		\item 
		The regular Borel subalgebras of derived length $3$ are of the following form:
		$$
		Lie(\mathfrak{t}_2,\underline{\ell}^{(a,b)},(a,b)), \text{ for some } (a,b)\in \ZZ^2\cup\{(0,-1),(-1,0)\} \text{ with }a\neq b.
		$$
	\end{enumerate}
\end{thm}
\begin{proof}
	Let $\mathfrak{B}$ be a regular Borel subalgebra.
%	Then $\mathfrak{B}^{(1)}$ is solvable.
Then by \eqref{PE0}, \eqref{PE01} and \eqref{PE1} (in Section~\ref{P}), we know that
\begin{center}
	$\mathfrak{B}= \text{Lie}(\mathfrak{t}_2,\mathcal{S}), \text{ for some } \mathcal{S}\subseteq \mathcal{M}\setminus\{(0,0)\},\, \mathcal{S}\neq \emptyset$
%\end{center}
and 
%\begin{center}
 $\mathfrak{B}^{(1)}$ = Lie($\mathcal{S}$) $\neq 0$.
\end{center}
%Therefore, $\mathfrak{B}$ is not Abelian.

\smallskip
\noindent
{\bf (1) : }If $\mathfrak{B}^{(1)}$ is Abelian, then $\mathfrak{B}$ is solvable of derived length $2$.

Otherwise, $\mathfrak{B}^{(1)}$ is a non-abelian and solvable regular subalgebra.
Hence, by \thref{GC2}, $\mathfrak{B}^{(1)}$ is solvable of derived length $2$. 
Therefore, $\mathfrak{B}$ is solvable of derived length $3$.

\smallskip
\noindent
{\bf (3) : }Let $\mathfrak{B}$ be a regular Borel subalgebra of derived length $3$.
Then $\mathfrak{B}^{(1)}$ is solvable of derived length $2$. 
Note that if $(a,b)\in \{(-1,m),(m,-1)\mid m\in \ZY\}$,	then $\underline{\ell}^{(a,b)}\cap \mathcal{M}= \emptyset$.
Therefore, by \thref{GC2}, it follows that 
\begin{center}
$\mathfrak{B}^{(1)}\subseteq $  Lie($\underline{\ell}^{(a,b)},(a,b)$), for some $(a,b)\in \ZZ^2\cup\{(0,-1),(-1,0)\}$ with $a\neq b$.
\end{center}
Now, $\mathfrak{B}$ is Borel; hence, we can conclude that
	$\mathfrak{B}$ =  Lie($\mathfrak{t}_2,\underline{\ell}^{(a,b)},(a,b)$).
	
\smallskip
\noindent
{\bf (2) : }
	Let $\mathfrak{B}$ be a regular Borel subalgebra of derived length $2$.
Then $\mathfrak{B}^{(1)}$ is Abelian.
Therefore, by \thref{RAS1}, it follows that 
\begin{center}
	$\mathfrak{B}^{(1)}\subseteq $  Lie($\{(0,-1),(-1,0)\}$) or
	$\mathfrak{B}^{(1)}\subseteq $  Lie($\underline{\ell}^{(a,b)}$), for some $(a,b)\in \mathcal{M}\setminus\{(0,0)\}$.
\end{center}
%\begin{center}
%or $\mathfrak{B}^{(1)}\subseteq$ Lie$(\Delta, \partial_{(1,-1)})$ or
%	$\mathfrak{B}^{(1)}\subseteq$ Lie$(\Delta, \partial_{(-1,1)})$. 
%\end{center}
Hence,
\begin{center}
	$\mathfrak{B}$ =  
	Lie($\mathfrak{t}_2,\{(0,-1),(-1,0)\}$) or $\mathfrak{B}$ = Lie($\mathfrak{t}_2,\underline{\ell}^{(a,b)}$), for some $(a,b)\in \mathcal{M}\setminus\{(0,0)\}$.
\end{center}
But, 
\begin{center}
Lie($\mathfrak{t}_2,(0,-1),(-1,0)$) $\subseteq$ Lie($\mathfrak{t}_2,(0,-1),\ell^{(-1,0)}$) =
Lie($\mathfrak{t}_2,(0,-1),\underline\ell^{(0,-1)}$),
\end{center}
 a Borel subalgebra by statement (3) and if $a\neq b$, then
\begin{center}
Lie($\mathfrak{t}_2,\underline{\ell}^{(a,b)}$) $\subseteq$ Lie($\mathfrak{t}_2,\underline{\ell}^{(a,b)}, (a,b)$), 	
\end{center}
a Borel subalgebra by statement (3).
Therefore, we are left with Lie($\mathfrak{t}_2,\underline{\ell}^{(a,a)}, (a,a)$), for $a\in \ZY$.
Note that $\underline{\ell}^{(a,a)}=\ell^{(a,a)}=\ell_{(1,1)}$ and $(0,0)\in \ell_{(1,1)}$; therefore, 
\begin{center}
	Lie($\mathfrak{t}_2,\underline{\ell}^{(a,a)}, (a,a)$) =
	Lie($\mathfrak{t}_2,{\ell}^{(a,a)}$) = Lie($\mathfrak{t}_2, \ell_{(1,1)}$) = Lie($\Delta, \ell_{(1,1)}$), for any $a\in \ZY$.
\end{center}
Let $(c,d)\in \mathcal{M}$ with $c\neq d$.
By \thref{P6}, we can conclude that Lie($\{(1,1),(c,d)\}$) is non-solvable; hence,  Lie($\mathfrak{t}_2,{\ell}_{(1,1)},(c,d)$) is non-solvable.
Thus, 
\begin{center}
	$\mathfrak{B}$ = Lie($\Delta, \ell_{(1,1)}$).
\end{center}
\end{proof}
Consider the lower triangular subalgebra $\mathfrak{j}_2^{+}$
and the upper triangular subalgebra $\mathfrak{j}_2^{-}$.
By \cite[Proposition~4.9]{AZ3} we know that $\mathfrak{j}^{\pm}$ are Borel subalgebras of Lie(Aut($\mathbb{A}^2$)).
Note that 
\begin{center}
	$\mathfrak{j}_2^{+}$ = Lie(JONQ$^{+}(\mathbb{A}^2)$) =
	Lie($\mathfrak{t}_2,{\ell}^{(-1,0)}, (0,-1)$) = Lie($\mathfrak{t}_2,\underline{\ell}^{(0,-1)}, (0,-1)$) 
\end{center}
and
\begin{center}
	$\mathfrak{j}_2^{-}$ = Lie(JONQ$^{-}(\mathbb{A}^2)$) = 
	Lie($\mathfrak{t}_2,{\ell}^{(0,-1)}, (-1,0)$) = Lie($\mathfrak{t}_2,\underline{\ell}^{(-1,0)}, (-1,0)$) .
\end{center}
Thus, \thref{MT1} provides an alternative way of proving that $\mathfrak{j}_2^{+}$ and $\mathfrak{j}_2^{-}$ are Borel subalgebras of Lie(Aut($\mathbb{A}^2$)).

Next, we draw the points of $\mathcal{M}$ (by teal ink), $(-1,-1)$ by blue ink and highlight the bidegrees (by red ink) of all the homogeneous derivations contained in $\mathfrak{j}_2^{+}$ and $\mathfrak{j}_2^{-}$ (in Fig~1 and Fig~2 respectively).
% in the first and second diagram respectively.
%\newpage
\begin{figure}[h]
	\begin{center}
		\tikzset{every picture/.style={line width=0.75pt}} %set default line width to 0.75pt        
		
		\begin{tikzpicture}[x=1pt,y=1pt,yscale=-1.25,xscale=1.25]
			\tikzstyle{conefill} = [fill=blue!20, draw = black!40]
			\tikzstyle{conefill_gamma} = [fill=red!20, draw = black!70]
			
			\coordinate (e1) at (20,0);
			\coordinate (e2) at (0,-20);
			
			%%%%%%%%%%%%%%%%% J_2^- %%%%%%%%%%%%
				
			\coordinate (O) at (210,230);
			\coordinate (Oxmax) at ($(O)+5*(e1)$);
			\coordinate (Oymax) at ($(O)+5*(e2)$);
			\coordinate (Ozmax) at ($(O)+5*(e1)-(e2)$);
			\coordinate (Oxmin) at ($(O)-2*(e1)$);
			\coordinate (Oymin) at ($(O)-2*(e2)$);
			\coordinate (Ozmin) at ($(O)-(e2)-2*(e1)$);
			\coordinate (Ozzmin) at ($(O)-(e2)$);
			\coordinate (T1) at ($(O)-3*(e2)+2*(e1)$);
			\coordinate (T2) at ($(O)-3.5*(e2)+2.5*(e1)$);
			
			\draw [->,color=black!90] (Oxmin) -- (Oxmax) node[below] {$x$-axis};
			\draw [->,color=black!90] (Oymin) -- (Oymax) node[right] {$y$-axis};
			\draw [->,color=black!90] (Ozmin) -- (Ozmax) node[right] {};
			\draw [->,color=red!90] (Ozzmin) -- (Ozmax) node[right] {};
			
			\foreach \x in {-1}
			\foreach \y in {-1}
			{
				\node[draw=blue!70,circle,inner sep=2pt,fill=blue!70] at ($(O)+\x*(e1)+\y*(e2)$) {};
			}
			\foreach \x in {-1,0,...,4}
			\foreach \y in {0,...,4}
			{
				\node[draw=teal!70,circle,inner sep=1.5pt,fill=teal!70] at ($(O)+\x*(e1)+\y*(e2)$) {};
			}
			\foreach \x in {0,...,4}
			{
				\node[draw=teal!70,circle,inner sep=1.5pt,fill=teal!70] at ($(O)+\x*(e1)-(e2)$) {};
			}
			\node[below left] at ($(O)-(e1)-(e2)$) {\tiny $(-1,-1)$};
			\node[below right] at ($(O)$) { $(0,0)$};
			\node[below left] at ($(O)-(e1)$) { $(-1,0)$};
			\node[below right] at ($(O)-(e2)$) { $\underline{\ell}^{(-1,0)}\cap \mathcal{M}$};
			\node[draw=red,circle,inner sep=2pt,fill=teal!70] at ($(O)$) {};
			\node[draw=red,circle,inner sep=2pt,fill=teal!70] at ($(O)-(e1)$) {};
			\node[draw=red,circle,inner sep=2pt,fill=teal!70] at ($(O)-(e2)+(e1)$) {};
			\node[draw=red,circle,inner sep=2pt,fill=teal!70] at ($(O)-(e2)+2*(e1)$) {};
			\node[draw=red,circle,inner sep=2pt,fill=teal!70] at ($(O)-(e2)+3*(e1)$) {};
			\node[draw=red,circle,inner sep=1.5pt,fill=teal!70] at ($(O)-(e2)+4*(e1)$) {};
			\node[draw=red,circle,inner sep=2pt,fill=teal!70] at ($(O)-(e2)$) {};
						
			\draw[black] (T1) node {Fig~2 : Upper triangular};
			\draw[black] (T2) node {subalgebra $\mathfrak{j}_2^{-}$};
			
			%%%%%%%%%%%%%%%%% J_2^+ %%%%%%%%%%%%
			
			\coordinate (O) at (20,230);
			\coordinate (Oxmax) at ($(O)+5*(e1)$);
			\coordinate (Oymax) at ($(O)+5*(e2)$);
			\coordinate (Ozmax) at ($(O)-(e1)+5*(e2)$);
			\coordinate (Oxmin) at ($(O)-2*(e1)$);
			\coordinate (Oymin) at ($(O)-2*(e2)$);
			\coordinate (Ozmin) at ($(O)-(e1)-2*(e2)$);
			\coordinate (Ozzmin) at ($(O)-(e1)$);
			\coordinate (T1) at ($(O)-3*(e2)+2*(e1)$);
			\coordinate (T2) at ($(O)-3.5*(e2)+2.5*(e1)$);
			
			\draw [->,color=black!90] (Oxmin) -- (Oxmax) node[below] {$x$-axis};
			\draw [->,color=black!90] (Oymin) -- (Oymax) node[right] {$y$-axis};
			\draw [->,color=black!90] (Ozmin) -- (Ozmax) node[right] {};
			\draw [->,color=red!90] (Ozzmin) -- (Ozmax) node[right] {};
			
			\foreach \x in {-1}
			\foreach \y in {-1}
			{
				\node[draw=blue!70,circle,inner sep=2pt,fill=blue!70] at ($(O)+\x*(e1)+\y*(e2)$) {};
			}
			\foreach \x in {-1,0,...,4}
			\foreach \y in {0,...,4}
			{
				\node[draw=teal!70,circle,inner sep=1.5pt,fill=teal!70] at ($(O)+\x*(e1)+\y*(e2)$) {};
			}
			\foreach \x in {0,...,4}
			{
				\node[draw=teal!70,circle,inner sep=1.5pt,fill=teal!70] at ($(O)+\x*(e1)-(e2)$) {};
			}
			\node[below left] at ($(O)-(e1)-(e2)$) {\tiny $(-1,-1)$};
			\node[below right] at ($(O)$) { $(0,0)$};
			\node[below left] at ($(O)-(e1)$) { $\underline{\ell}^{(0,-1)}\cap \mathcal{M}$};
			\node[below right] at ($(O)-(e2)$) { $(0,-1)$};
			\node[draw=red,circle,inner sep=2pt,fill=teal!70] at ($(O)$) {};
			\node[draw=red,circle,inner sep=2pt,fill=teal!70] at ($(O)-(e1)$) {};
			\node[draw=red,circle,inner sep=2pt,fill=teal!70] at ($(O)-(e1)+(e2)$) {};
			\node[draw=red,circle,inner sep=2pt,fill=teal!70] at ($(O)-(e1)+2*(e2)$) {};
			\node[draw=red,circle,inner sep=2pt,fill=teal!70] at ($(O)-(e1)+3*(e2)$) {};
			\node[draw=red,circle,inner sep=2pt,fill=teal!70] at ($(O)-(e1)+4*(e2)$) {};
			\node[draw=red,circle,inner sep=2pt,fill=teal!70] at ($(O)-(e2)$) {};
		
			\draw[black] (T1) node {Fig~1 : Lower triangular};
			\draw[black] (T2) node {subalgebra $\mathfrak{j}_2^{+}$};
			
		\end{tikzpicture}
	\end{center}
\end{figure}

%Thus we have an alternate way of showing that $\mathfrak{j}_2^{+}$ and $\mathfrak{j}_2^{-}$ are Borel subalgebras of Lie(Aut($\mathbb{A}^2$)) (One proof is given in \cite[Proposition~4.9]{AZ3}).
%\noindent
We know that $\Delta,\partial_{(0,0)}$ are contained in every regular Borel subalgebra of Lie(Aut($\mathbb{A}^2$)).
Consider $(a,b)\in \mathcal{M}\setminus\{(0,0)\}$.
Below we write down all the regular Borel subalgebras containing $\partial_{(a,b)}$.
\begin{enumerate}[\rm(a)]
	
	\item  If $a=b$, then the only regular Borel subalgebra containing $\partial_{(a,a)}$ is Lie($\Delta,\ell_{(1,1)}$).
	
	\item  If $(a,b)\in \ZZ^2\cup\{(0,-1),(-1,0)\}$ with $a\neq b$, then there are countably infinite number of regular Borel subalgebras  containing $\partial_{(a,b)}$; precisely, the following:
	\begin{center}
		Lie($\mathfrak{t}_2,\underline{\ell}^{(a,b)},(a,b)$) and Lie($\mathfrak{t}_2, \ell^{(a,b)}, (c,d)$), for any $(c,d)\in \widetilde{\ell}_{(a,b)}\cap \mathcal{M}\setminus\{(0,0)\}$.
	\end{center}
	\item Let $m\in \ZY$.
	Then there are countably many regular Borel subalgebras  containing $\partial_{(m,-1)}$; precisely, the following:
	$$\begin{array}{lll}
		& &\text{Lie}(\mathfrak{t}_2, \ell^{(m,-1)}, (c,d)), \text{ for any } (c,d)\in \widetilde{\ell}_{(m,-1)}\cap \mathcal{M}\setminus\{(0,0)\}\\ &=&
		\text{Lie}(\mathfrak{t}_2, \ell^{(0,-1)}, (c,0)), \text{ for any } c\in \ZY\cup\{-1\}.	
	\end{array}$$
	Similarly, there are countably many regular Borel subalgebras  containing $\partial_{(-1,m)}$.
	Precisely the following:
	$$\begin{array}{lll}
		%	& &\text{Lie}(\mathfrak{t}_2, \ell^{(m,-1)}, (c,d)), \text{ for any } (c,d)\in \widetilde{\ell}_{(m,-1)}\cap \mathcal{M}\\
		%	&=& 
		\text{Lie}(\mathfrak{t}_2, \ell^{(-1,0)}, (0,d)), \text{ for any } d\in \ZY\cup\{-1\}.	
	\end{array}$$
	%	}
	\end{enumerate}
\begin{rem}\thlabel{MT3}
{\rm
	Note that for any $(c,d)\in \ZZ\setminus\{(0,0)\}$, $\partial_{(c,d)}$ is a non-locally finite homogeneous derivation.
	Hence, by \cite[Theorem~3.4]{AZ3}, the regular Borel subalgebras Lie($\Delta, \ell_{(1,1)}$), Lie($\mathfrak{t}_2,\underline{\ell}^{(a,b)},(a,b)$), with $(a,b)\in \ZZ$ and $a\neq b$ does not correspond to any Borel subgroup of Aut($\mathbb{A}^2$).
}	
\end{rem}
	
Next, we note down a result concerning the non-regular Borel subalgebras.
%By \cite[Corollary~2.3]{AZ3}, we know that every derivation of Lie(Aut($\mathbb{A}^2$)) is contained in a Borel subalgebra of Lie(Aut($\mathbb{A}^2$)). 
%  Concerning the non-regular Borel subalgebraswe note down an observation below.
\begin{lem}\thlabel{MT2}
	Suppose $\delta=\underset{(a,b)\in \mathcal{S}}{\sum}\delta_{(a,b)}$ is the weight decomposition of~$\delta\in$ Lie(Aut($\mathbb{A}^2$)), where $\mathcal{S}\subseteq \mathcal{M}$ and  $\delta_{(a,b)}$ is the non-zero homogeneous component of bidegree $(a,b)$.
	
	If there exist $(c,d), (e,f)\in \mathcal{S}\setminus\{(0,0)\}$ such that $(c,d)\notin {\ell}^{(e,f)}$, $(c,d)\notin \underline{\ell}^{(e,f)}$ and $(e,f)\notin \underline{\ell}^{(c,d)}$.
	Then every Borel subalgebra 
	containing $\delta$ is non-regular.
\end{lem}
\begin{proof}
	Let $\mathfrak{L}$ be a regular subalgebra of Lie(Aut($\mathbb{A}^2$)) containing $\delta$.
	Then $\partial_{(c,d)},\partial_{(e,f)}\in~\mathfrak{L}.$
	Now by \thref{LNS1}, we can conclude that Lie($\{(c,d),(e,f)\}$) is non-solvable; hence, $\mathfrak{L}$ is non-solvable.
	
	By \cite[Corollary~2.3]{AZ3}, we know that every derivation of Lie(Aut($\mathbb{A}^2$)) is contained in a Borel subalgebra of Lie(Aut($\mathbb{A}^2$)).
	Let $\mathfrak{B}$ be a Borel subalgebra of Lie(Aut($\mathbb{A}^2$)) containing $\delta$.
	Then  $\mathfrak{B}$ is a non-regular Borel subalgebra of Lie(Aut($\mathbb{A}^2$)).
\end{proof}
%\begin{rem}
%	{\rm
	Consider the derivation $\partial_{(0,1)}+\partial_{(1,0)}$.
	Note that here, $\mathcal{S}=\{(0,1),(1,0)\}$ and $(0,1)\notin~\ell^{(1,0)}$, $(0,1)\notin \underline{\ell}^{(1,0)}$, $(1,0)\notin \underline{\ell}^{(0,1)}$.
	Therefore, by the above result (\thref{MT2}) every Borel subalgebra 
	containing $\partial_{(0,1)}+\partial_{(1,0)}$ is non-regular.
	%	}
%\end{rem}

\subsection{Isomorphism classes}\label{IC}
In this subsection, we discuss the isomorphism classes of the regular Borel subalgebras.
%The first result 
% and give a complete list of non-isomorphic regular Borel subalgebras (\thref{Rm6}).
%At first we state a lemma about the cardinality of the generating set of the regular Borel subalgebras.
%At first we describe some abstract Lie algebras and discuss there isomorphism classes.
%
%Let $n\in \ZZ\cup\{-1\}$ and $\lambda=(\lambda_i), \alpha=(\alpha_i), \gamma=(\gamma_i)\in$ Map($\ZY,\kk$).
%We define a Lie algebra structure $\mathscr{G}(n,\lambda,\alpha,\gamma)$ on the vector space 
%$V=\kk Y_1\oplus \kk Y_2\oplus(\overunderset{\infty}{i=0}{\oplus}\kk X_i)$ as follows:
%\begin{center}
%	$[Y_1,X_0]=X_0=[Y_2,X_0]$, $[Y_1,X_i]=\lambda_i X_i$, $[Y_2,X_i]=\alpha_i X_i$, $[X_0,X_i]=\gamma_i X_{i+n}$, for all $i\geqslant 1$. 
%\end{center}
At first, we observe that every regular Borel subalgebra of Lie(Aut($\mathbb{A}^2$)) is of the form  $\mathscr{G}(n,\mu,\pmb{\lambda},\pmb{\alpha},\pmb{\gamma})$,
for some $n\in \ZZ\cup\{-1\}$, $\mu\in \{0,1\}$ and $\pmb{\lambda}, 
\pmb{\alpha}, \pmb{\gamma}\in$ Map($\ZY,\kk$) (for notation refer to Section~\ref{N}, notation 21) as stated in the structure theorem, in Section~\ref{I}.
\begin{lem}\thlabel{IC0}
	\begin{enumerate}[\rm(a)]
		\item Lie$(\Delta,\ell_{(1,1)})$= $\mathscr{G}(0,0,(i+1)_i,\pmb{0},\pmb{0})$.
		\item 
%		$\mathfrak{j}_2^{\pm}$ 
%		(or $\mathfrak{j}_2^{-}$)=
Lie$(\mathfrak{t}_2,\underline{\ell}^{(0,-1)}, (0,-1))$ (or Lie$(\mathfrak{t}_2,\underline{\ell}^{(-1,0)}, (-1,0))$)
%		(or Lie$(\mathfrak{t}_2,\underline{\ell}^{(-1,0)}, (-1,0))$)
$= \mathscr{G}(-1,1,(2-i)_i,(-i)_i,(i-1)_i).$
		\item Lie$(\mathfrak{t}_2,\underline{\ell}^{(a,b)},(a,b))=$
			$\mathscr{G}(\gcd(a,b),1,\pmb{\lambda},\pmb{\alpha},\pmb{\gamma})$, with $(a,b)\in \ZZ^2$ with $a\neq b$, where $$\lambda_i=\frac{p+q}{a+b}+\frac{i-1}{\gcd(a,b)},\, 
			\alpha_i=\frac{p-q}{a-b}+\frac{i-1}{\gcd(a,b)},\,
			\gamma_i= (p-q)+\frac{(i-1)(a-b)}{\gcd(a,b)}
			%\frac{1}{d}\begin{vmatrix}
				%d(p_1+1)+(i-1)a & a+1\\
				%d(q_1+1)+(i-1)b& b+1\\
			%\end{vmatrix}
			$$
			and $(p,q)$ is the  minimum element of $\underline{\ell}^{(a,b)}\cap \mathcal{M}$ with respect to the lexicographic ordering.
%			 prioritizing the first component.	
%		\end{center}
	\end{enumerate}
\end{lem}
\begin{proof}
{\bf (a) : }Putting $Y_1=\frac{1}{2}\Delta$, $Y_2=\partial_{(0,0)}$ and $X_i=\partial_{(i+1,i+1)}$, for all $i\in \ZZ$ in the definition of $\mathscr{G}$ (Section~\ref{N}, notation 21) and with the help
commutation relations \eqref{CR1}~and~\eqref{CR3}, we get our result.

\smallskip
{\bf (b) : } Putting $Y_1=-\Delta$, $Y_2=\partial_{(0,0)}$, $X_0=-\partial_{(0,-1)}$ and $X_i=(1/i)\partial_{(-1,i-1)}$, for all $i\in \ZY$ in the definition of $\mathscr{G}$ (Section~\ref{N}, notation 21) and with the help
commutation relations \eqref{CR1} and \eqref{CR3}, we get that \begin{center}
	Lie$(\mathfrak{t}_2,\underline\ell^{(0,-1)}, (0,-1))$ = Lie$(\mathfrak{t}_2,\ell^{(-1,0)}, (0,-1))$ =
$ \mathscr{G}(-1,1,(2-i)_i,(-i)_i,(i-1)_i).$
\end{center}

Similarly, by putting $Y_1=-\Delta$, $Y_2=-\partial_{(0,0)}$, $X_0=\partial_{(-1,0)}$ and $X_i=(-1/i)\partial_{(i-1,-1)}$, for all $i\in \ZY$ in the definition of $\mathscr{G}$ (Section~\ref{N}, notation 21) and with the help
commutation relations \eqref{CR1} and \eqref{CR3}, we get that \begin{center}
	Lie$(\mathfrak{t}_2,\underline\ell^{(-1,0)}, (-1,0))$ = Lie$(\mathfrak{t}_2,\ell^{(0,-1)}, (-1,0))$ =
$ \mathscr{G}(-1,1,(2-i)_i,(-i)_i,(i-1)_i).$
\end{center}

\smallskip
{\bf (c) : }
Let $d=\gcd(a,b)$.
	By \thref{N2}, there exists a minimum element $(p,q)$ of $\underline{\ell}^{(a,b)}\cap \mathcal{M}$ with respect to the lexicographic ordering and 
$$\underline{\ell}^{(a,b)}\cap \mathcal{M}=\{\bigg(p+\dfrac{ia}{d},q+\dfrac{ib}{d}\bigg)\mid i\in \ZZ\}.$$
Let $X_0=\partial_{(a,b)}$ and $X_{i+1}=\partial_{(p+{ia}/{d},q+{ib}/{d})}$, for all $i\in \ZZ$.
Then, by commutation relation \eqref{CR1}, we have
%the commutation relations in Section~\ref{S0}, 
%it follows that
$$[X_0,X_i]=\begin{vmatrix}
	p+\frac{(i-1)a}{d}+1 & a+1\\
	q+\frac{(i-1)b}{d}+1 & b+1\\
\end{vmatrix}X_{i+d}=\bigg((p-q)+\frac{(i-1)(a-b)}{d}\bigg)X_{i+d}, \text{ by }(p,q)\in \underline{\ell}^{(a,b)}.$$
Next, putting $Y_1=\frac{1}{a+b}\Delta$, $Y_2=\frac{1}{a-b}\partial_{(0,0)}$ and $X_i,$ for all $i\in \ZZ$ (as defined above) in the definition of $\mathscr{G}$ (Section~\ref{N}, notation 21) and with the help
commutation relations \eqref{CR1}~and~\eqref{CR3}, we get that
\begin{center}
	Lie($\mathfrak{t}_2,\underline\ell^{(a,b)}, (a,b)$)	
	%Lie($\mathfrak{t}_2,\ell^{(0,-1)}, (n,0)$) 
	= $\mathscr{G}\big(d,1,(\lambda)_i, (\alpha)_i, (\gamma)_i\big)$, 
\end{center}
where $$\lambda_i=\frac{p+q}{a+b}+\frac{i-1}{d},\, 
\alpha_i=\frac{p-q}{a-b}+\frac{i-1}{d},\,
\gamma_i=(p-q)+\frac{(i-1)(a-b)}{d}.$$
\end{proof}
%Now we turn our attention towards the regular Borel subalgebras of Lie(Aut($\mathbb{A}^2$)).
By \thref{MT1}, Lie($\Delta$, $\ell_{(1,1)}$) is of derived length $2$ and all other regular Borel subalgebras are of derived length $3$.
Therefore, Lie($\Delta$, $\ell_{(1,1)}$) is not isomorphic to any other regular Borel subalgebra.
So we need to classify 
%the isomorphism class of 
all the regular Borel subalgebras of derived length $3$, up to isomorphism.
%Moreover, Lie$(\mathfrak{t}_2,\underline{\ell}^{(0,-1)}, (0,-1))$, Lie$(\mathfrak{t}_2,\underline{\ell}^{(-1,0)}, (-1,0))$ are infinitely generated and isomorphic (cf. \thref{IC0}(b)).
%So we need to classify all the Lie$(\mathfrak{t}_2,\underline{\ell}^{(a,b)},(a,b))$
%with $(a,b)\in \ZZ^2$ with $a\neq b$.

At first, we prove a result about isomorphism classes of $\mathscr{G}(n,1,\pmb{\lambda},\pmb{\alpha},\pmb{\gamma})$.
%
% 

% From the commutation relations in Section~\ref{S0} it follows that
% \begin{center}
% 	Lie($\mathfrak{t}_2,\underline\ell^{(n,0)}, (n,0)$)=	Lie($\mathfrak{t}_2,\ell^{(0,-1)}, (n,0)$) $\cong\mathscr{G}\big(n,(\frac{i-2}{n})_i, (\frac{i}{n})_i, (n+i)_i\big)$,  $n\in \ZY\cup\{-1\}$.
% \end{center}
\begin{lem}\thlabel{IC1}
	Let $m,n\in \ZY\cup\{-1\}$ and $\pmb\lambda=(\lambda_i)_i,\, \pmb\mu=(\mu_i)_i,\, \pmb\alpha=(\alpha_i)_i,\, \pmb\beta=(\beta_i)_i,\\ \pmb\gamma=(\gamma_i)_i,\, \pmb\epsilon=(\epsilon_i)_i \in$ Map$(\ZY,\kk)$.
Then the following statements hold:
\begin{enumerate}[\rm(a)]
	\item if $m\neq n$ then $\mathscr{G}(m,1,\pmb\mu,\pmb\beta,\pmb\epsilon)$ and $ \mathscr{G}(n,1,\pmb\lambda,\pmb\alpha,\pmb\gamma)$ are not graded isomorphic.
	\item if $m=n$, $\gamma_i,\epsilon_i\in \kk^{*}$, $\lambda_i\neq \alpha_i$, $\mu_i\neq \beta_i$; for all $i\in \ZY$ and $\bigg(\dfrac{\lambda_i-\beta_i}{\mu_i-\beta_i}\bigg),\bigg(\dfrac{\mu_i-\alpha_i}{\mu_i-\beta_i}\bigg)$ does not depend on $i$, then $\mathscr{G}(m,1,\pmb\mu,\pmb\beta,\pmb\epsilon)\cong \mathscr{G}(n,1,\pmb\lambda,\pmb\alpha,\pmb\gamma)$.
\end{enumerate}	
\end{lem}
\begin{proof}
	{\bf (a) : }If $n=-1$, then for any $r\in \ZZ$,  $\{Y_1,Y_2,X_i\mid i\geqslant r\}$ is a generating set of $\mathscr{G}(-1,1,\pmb\lambda,\pmb\alpha,\pmb\gamma)$.
	%	Moreover, it follows from the definition of Lie bracket in this case that
	Therefore, $\mathscr{G}(-1,1,\pmb\lambda,\pmb\alpha,\pmb\gamma)$ is infinitely generated.
	
	Note that for any $n\in \ZY$, $\{Y_1,Y_2,X_0, \dots, X_{n-1}\}$ is the smallest generating set of $\mathscr{G}(n,1,\pmb\lambda,\pmb\alpha,\pmb\gamma)$, consisting of regular elements.
	Therefore, if $m\neq n$ then $\mathscr{G}(m,1,\pmb\mu,\pmb\beta,\pmb\epsilon)$ is not graded isomorphic to $\mathscr{G}(n,1,\pmb\lambda,\pmb\alpha,\pmb\gamma)$.
	
	\smallskip
	\noindent
	{\bf (b) : } According to Section~\ref{N} notation~$21$, the Lie algebra structure  
 $\mathscr{G}(n,1,\pmb\lambda,\pmb\alpha,\pmb\gamma)$ on the vector space 
	$\kk Y_1\oplus \kk Y_2\oplus\bigg(\overunderset{\infty}{i=0}{\oplus}\kk X_i\bigg)$ is defined as follows:
	\begin{center}
		$[Y_1,X_0]=X_0=[Y_2,X_0]$, $[Y_1,X_i]=\lambda_i X_i$, $[Y_2,X_i]=\alpha_i X_i$, $[X_0,X_i]=\gamma_i X_{i+n},\, \forall i\geqslant 1,$ 
	\end{center}
	and the Lie algebra structure  $\mathscr{G}(n,1,\pmb\mu,\pmb\beta,\pmb\epsilon)$ on the vector space 
	$\kk Y_1\oplus \kk Y_2\oplus\bigg(\overunderset{\infty}{i=0}{\oplus}\kk W_i\bigg)$ is defined as follows:
	\begin{center}
		$[Y_1,W_0]=W_0=[Y_2,W_0]$, $[Y_1,W_i]=\mu_i W_i$, $[Y_2,W_i]=\beta_i W_i$, $[X_0,W_i]=\epsilon_i W_{i+n},\, \forall i\geqslant 1$. 
	\end{center}	
	Let $a_{11}=\bigg(\dfrac{\lambda_1-\beta_1}{\mu_1-\beta_1}\bigg)$,  $a_{22} = \bigg(\dfrac{\mu_1-\alpha_1}{\mu_1-\beta_1}\bigg)$ and $s=a_{11}+a_{22}-1=\bigg(\dfrac{\lambda_1-\alpha_1}{\mu_1-\beta_1}\bigg)$.
%	Then $a_{11},\, a_{22}$ and $s$ are all constants, not depending on $i$.
	Then $s\neq 0,$ since $\lambda_1\neq \alpha_1$.
	Next, let $b_{11}=\dfrac{a_{22}}{s}=\bigg(\dfrac{\mu_1-\alpha_1}{\lambda_1-\alpha_1}\bigg)$ and $b_{22}=\dfrac{a_{11}}{s}=\bigg(\dfrac{\lambda_1-\beta_1}{\lambda_1-\alpha_1}\bigg)$.
	Then the maps
$$\phi :  \mathscr{G}(n,1,\pmb\lambda,\pmb\alpha,\pmb\gamma)\rightarrow \mathscr{G}(n,1,\pmb\mu,\pmb\beta,\pmb\epsilon) \text{ defined by:}$$
		$$\begin{array}{cccc}
		\phi(Y_1)&=&a_{11}Y_1+(1-a_{11})Y_2 &\\ \phi(Y_1)&=&(1-a_{22})Y_1+a_{22}Y_2&\\ 
	    \phi(X_i)&=&W_i, & 0\leqslant i\leqslant n
%	\end{center}	
	\end{array}$$
and
%	 extended with help of Lie brackets
%	is an isomorphism, because
	  $$\varphi :  \mathscr{G}(n,1,\pmb\mu,\pmb\beta,\pmb\epsilon)\rightarrow \mathscr{G}(n,1,\pmb\lambda,\pmb\alpha,\pmb\gamma) \text{ defined by:}$$
	$$\begin{array}{cccc}
		\varphi(Y_1) &=&b_{11}Y_1+(1-b_{11})Y_2&\\ \varphi(Y_1)&=&(1-b_{22})Y_1+b_{22}Y_2&\\
        \varphi(W_i)&=&X_i,  &0\leqslant i\leqslant n
	\end{array}$$
	are Lie algebra homomorphisms. 
	Note that $\phi^{-1}=\varphi$.
	Hence, the result follows. 
\end{proof}

%Note that $\underline\ell^{(n,0)}=\ell^{(0,-1)}$ for all $n\in \ZY\cup\{-1\}$.
Recall that $\underline{\ell}^{(n,0)}=\ell^{(0,-1)}$, for all $n\in \ZY\cup\{-1\}$.
Therefore,
\begin{center}
Lie($\mathfrak{t}_2,\underline{\ell}^{(n,0)},(n,0)$) = Lie($\mathfrak{t}_2,{\ell}^{(0,-1)},(n,0)$).
\end{center}
Consider the set $\mathscr{A}:=\{Lie(\mathfrak{t}_2,\ell^{(0,-1)}, (n,0))\mid n\in \ZY\cup \{-1\}\}$ of regular Borel subalgebras of derived length $3$.
Now, we prove the main result of this subsection.
\begin{prop}\thlabel{IC2}
The following statements hold:
\begin{enumerate}[\rm(a)]
	\item For any $(a,b)\in \ZZ\cup \{(0,-1),(0,-1)\}$ with $a\neq b$,
		$$\begin{array}{lll}
	Lie(\mathfrak{t}_2,\underline\ell^{(a,b)}, (a,b))\cong \biggl\{\begin{aligned}
			&\hspace{9 mm}Lie(\mathfrak{t}_2,\ell^{(0,-1)}, (-1,0))\in \mathscr{A}, \text{ if }(a,b)\in \{(0,-1),(-1,0)\} \\
			&Lie(\mathfrak{t}_2,\ell^{(0,-1)},(\gcd(a,b),0))\in \mathscr{A}, \text{ if }(a,b)\in \ZZ^2 \text{ with }a\neq b.
		\end{aligned}	
	\end{array}$$
	\item  Two distinct elements of $\mathscr{A}$ are not graded isomorphic to one another.
\end{enumerate}	
\end{prop}
\begin{proof}
	Note that from \thref{IC0}(c), it follows that, for any $ n\in \ZY$,
	\begin{equation}\label{ICE1}
		\text{Lie}(\mathfrak{t}_2,\underline\ell^{(n,0)}, (n,0))=	
		\text{Lie}(\mathfrak{t}_2,\ell^{(0,-1)}, (n,0)) =\mathscr{G}(n,1,({(i-2)}/{n})_i,({i}/{n})_i,(i)_i).
	\end{equation}
{\bf (a) : }If $(a,b)\in \{(0,-1),(-1,0)\}$ then by
\thref{IC0}(b),
 we know that% the commutation relations \eqref{E19}, \eqref{E17}, \eqref{E18}, \eqref{E29} and \eqref{E32}, we can conclude that 
\begin{center}Lie($\mathfrak{t}_2,\ell^{(-1,0)}, (0,-1)$)
%	Lie($\mathfrak{t}_2,\underline\ell^{(-1,0)}, (-1,0)$)=	
 $=\mathscr{G}(-1,1,(2-i)_i, (-i)_i, (i-1)_i)$= Lie($\mathfrak{t}_2,\ell^{(0,-1)}, (-1,0)$) $\in\mathscr{A}$.
\end{center}

Let $(a,b)\in \ZZ^2$ with $a\neq b$ and $d:=\gcd(a,b)\in \ZY$.
Then by \thref{IC0}(c), we  know that
\begin{center}
		Lie($\mathfrak{t}_2,\underline\ell^{(a,b)}, (a,b)$)	
%Lie($\mathfrak{t}_2,\ell^{(0,-1)}, (n,0)$) 
$=\mathscr{G}\big(d,1,(\lambda)_i, (\alpha)_i, (\gamma)_i\big)$, where 
\end{center}
 $\lambda_i=\dfrac{p+q}{a+b}+\dfrac{i-1}{d},\, 
\alpha_i=\dfrac{p-q}{a-b}+\dfrac{i-1}{d},\,
\gamma_i= (p-q)+\dfrac{(i-1)(a-b)}{d}$
and $(p,q)$ is the  minimum element of $\underline{\ell}^{(a,b)}\cap \mathcal{M}$ with respect to the lexicographic ordering.
Now $a\neq b$; hence, $\ell^{(a,b)}\cap\underline\ell^{(a,b)}\cap\mathcal{M}=\emptyset.$
Therefore, $\gamma_i\in \kk^{*}$, for all $i\in \ZY$.
Moreover, from \eqref{ICE1}, we know that
\begin{center}
	Lie($\mathfrak{t}_2,\underline\ell^{(d,0)}, (d,0)$)=	
	Lie($\mathfrak{t}_2,\ell^{(0,-1)}, (d,0)$) 
	$=\mathscr{G}\big(d,1,(\mu)_i, (\beta)_i, (\epsilon)_i\big)$, where 
\end{center}
$$\mu_i={(i-2)}/{d},\, 
\beta_i={i}/{d},\,
\epsilon_i= i.$$
%Next, we will use \thref{IC1}(b) and show that Lie($\mathfrak{t}_2,\underline\ell^{(a,b)}, (a,b)$) $\cong$ Lie($\mathfrak{t}_2,\ell^{(0,-1)}, (d,0)$).

Now, we check all the assumptions of \thref{IC1}(b) for $\mathscr{G}\big(d,1,\pmb{\lambda}, \pmb{\alpha}, \pmb{\gamma}\big)$ and $\mathscr{G}\big(d,1,\pmb\mu, \pmb\beta, \pmb\epsilon\big)$.
First,
$$\lambda_i-\alpha_i=\frac{2(aq-bp)}{(a+b)(a-b)}=\frac{-2}{a+b}\in \kk^{*},\, \text{ since }(p,q)\in \underline{\ell}^{(a,b)}.$$
%
% $(p_1,q_1)\in \underline{\ell}^{(a,b)}$ and $(a,b)\notin\underline{\ell}^{(a,b)} $, hence $\gamma_i\neq 0$, for all $i$.
%
%
%
Next, 
$$\mu_i-\beta_i=\frac{i-2}{d}-\frac{i}{d}=\frac{-2}{d}\in \kk^{*},$$
$$\dfrac{\lambda_i-\beta_i}{\mu_i-\beta_i}=\dfrac{a+b-d(p+q)}{2(a+b)},
\text{ does not depend on } i
$$
and
$$\dfrac{\mu_i-\alpha_i}{\mu_i-\beta_i}=\dfrac{a-b+d(p-q)}{2(a-b)},
\text{ does not depend on } i.$$
Therefore, by \thref{IC1}(b), we have	$$\text{Lie}(\mathfrak{t}_2,\underline\ell^{(a,b)}, (a,b))\cong \text{Lie}(\mathfrak{t}_2,\ell^{(0,-1)},(d,0))=
\text{Lie}(\mathfrak{t}_2,\ell^{(0,-1)},(\gcd(a,b),0)).$$
\smallskip
\noindent
{\bf (b) : }
	From 
\thref{IC0}(b), we know that
\begin{center}
	Lie$(\mathfrak{t}_2,\underline{\ell}^{(-1,0)}, (-1,0))$ =
	Lie$(\mathfrak{t}_2,{\ell}^{(0,-1)}, (-1,0))$
	%		(or Lie$(\mathfrak{t}_2,\underline{\ell}^{(-1,0)}, (-1,0))$)
	$= \mathscr{G}(-1,1,(2-i)_i,(-i)_i,(i-1)_i)$
	\end{center}
and from \eqref{ICE1}, for any $n\in \ZY$, we have
\begin{equation*}
	\text{Lie}(\mathfrak{t}_2,\underline\ell^{(n,0)}, (n,0))=	\text{Lie}(\mathfrak{t}_2,\ell^{(0,-1)}, (n,0)) =\mathscr{G}(n,1,({(i-2)}/{n})_i,({i}/{n})_i,(i)_i).
\end{equation*}
Now by \thref{IC1}(a), the result follows.
\end{proof}
\begin{rem}\thlabel{Rm6}
	{\rm
		From the above result (Proposition~\ref{IC2}), we get a list of regular Borel subalgebras up to isomorphism; precisely, the following:
\begin{center}
Lie($\Delta,\ell_{(1,1)}$) and Lie($\mathfrak{t}_2,\ell^{(0,-1)},(a,0)$) for $a\in \ZY\cup\{-1\}$.	
\end{center}
%All other subalgebras are isomorphic to one of them.
	}
\end{rem}
\noindent
It is clear that they are not graded isomorphic to each other (cf. \thref{IC2}(b)).
Moreover,
% Lie$(\mathfrak{t}_2,\underline{\ell}^{(0,-1)}, (0,-1))$ and 
Lie$(\mathfrak{t}_2,{\ell}^{(0,-1)}, (-1,0))$ is infinitely generated.
So 
Lie$(\mathfrak{t}_2,{\ell}^{(0,-1)}, (-1,0))$ is not isomorphic to Lie$(\mathfrak{t}_2,{\ell}^{(0,-1)},(a,0))$, for $a\in \ZY$.
%with $(a,b)\in \ZZ^2$ with $a\neq b$.
But we do not know whether for $a,c~\in~\ZY$ with $a\neq c$, the  subalgebra Lie($\mathfrak{t}_2,\ell^{(0,-1)},(a,0)$) is isomorphic to Lie($\mathfrak{t}_2,\ell^{(0,-1)},(c,0)$) or not.

\medskip
\noindent		
{\bf Acknowledgements.}
The author is grateful to Ivan Arzhantsev for introducing her to Lie theory, suggesting to study the regular Borel subalgebras of Lie(Aut($\mathbb{A}^2$)), carefully going through the earlier draft and giving valuable suggestions.
The author is also thankful to Yulia Zaitseva for useful discussions and Ivan Beldiev for pointing out several misprints. 
% The author acknowledges the HSE University.

This article is an output of a research project (HSE-BR-2025-22) implemented as part of the Basic Research Program at HSE University.

			\end{document}